\documentclass[a4paper]{article}
\usepackage[T1]{fontenc}
\usepackage{lmodern}
\usepackage[english]{babel}
\usepackage[utf8]{inputenc}
\usepackage{setspace}
\usepackage{hyperref}
\usepackage{microtype}
\usepackage{quoting}
\usepackage{amsmath}
\usepackage{amssymb}
\usepackage{comment}
\usepackage{amsthm}
\usepackage{subcaption}
\usepackage{mathrsfs}
\usepackage{booktabs}
\usepackage{caption}
\usepackage{graphicx}
\usepackage{multirow}
\usepackage{yhmath}
\usepackage{xcolor}
\usepackage[affil-it]{authblk}
\usepackage{mathtools}
\usepackage{tabularray}
\usepackage{longtable}
\usepackage{comment}
\usepackage[misc]{ifsym}
\usepackage{url}
\usepackage{authblk}

\newcommand{\bvec}{\mathbf b}
\newcommand{\cvec}{\mathbf c}
\newcommand{\dvec}{\mathbf d}
\newcommand{\evec}{\mathbf e}
\newcommand{\fvec}{\mathbf f}
\newcommand{\lvec}{\mathbf l}
\newcommand{\qvec}{\mathbf q}
\newcommand{\rvec}{\mathbf r}
\newcommand{\uvec}{\mathbf u}
\newcommand{\vvec}{\mathbf v}
\newcommand{\xvec}{\mathbf x}

\newcommand{\xl}{\mathbf {x^l}}
\newcommand{\xu}{\mathbf {x^u}}
\newcommand{\yvec}{\mathbf y}
\newcommand{\zl}{\mathbf {z^l}}
\newcommand{\zu}{\mathbf {z^u}}
\newcommand{\0}{\mathbf 0}

\newcommand{\dx}{\boldsymbol{\Delta x}}
\newcommand{\dxl}{\boldsymbol{\Delta x^l}}
\newcommand{\dxu}{\boldsymbol{\Delta x^u}}
\newcommand{\dy}{\boldsymbol{\Delta y}}
\newcommand{\dzl}{\boldsymbol{\Delta z^l}}
\newcommand{\dzu}{\boldsymbol{\Delta z^u}}

\newcommand{\cpp}{C\texttt{++} }

\newtheorem{assumption}{Assumption}
\newtheorem{proposition}{Proposition}

\title{A factorisation-based regularised interior point method using the augmented system}
\author[1]{Filippo Zanetti\footnote{f.zanetti@ed.ac.uk}}
\author[1]{Jacek Gondzio}
\affil[1]{School of Mathematics, University of Edinburgh, Edinburgh, UK}

\date{}

\begin{document}
\maketitle

\begin{abstract}
\noindent
This paper describes the implementation of a new interior point solver for linear programming for the open-source optimization library HiGHS. The solver uses a direct factorisation to solve the Newton systems, choosing the best approach between the normal equations and augmented system. Details of the multifrontal factorisation routine are given, with attention to the features that allow to achieve high performance, like storage formats, use of efficient dense linear algebra subroutines and parallelism. The paper also describes the use of pivoting and regularisation strategies to ensure that a stable factorisation is obtained, despite the ill-conditioning of the matrices. Results on three different collections of problems are presented which highlight the improved performance of the solver compared to the existing HiGHS interior point method.
\end{abstract}


\section{Introduction}
\label{intro}
HiGHS is the world's best open source linear optimization library. It includes a simplex \cite{HuaHal:highs} and interior point \cite{SchGon:basisIPM} solvers for linear programming, a branch-and-cut solver for mixed integer linear programs, and an active set method for quadratic programs.

The interior point solver, called IPX \cite{SchGon:basisIPM}, uses a Krylov subspace iterative solver with a basis preconditioner to solve the linear systems of normal equations at each interior point iteration. This strategy allows problems to be solved without actually forming the normal equations matrix and without having to compute or store its factorisation, which in some cases could be very expensive. 
IPX builds and updates a maximum volume basis for the scaled constraint matrix and uses it to construct an efficient preconditioner for the normal equations matrix. The construction of such basis leverages simplex-like techniques and, when an optimal solution is approached, identifies a close-to-optimal basis, which in turn facilitates crossover to find an optimal vertex. Employing a preconditioner based on maximum volume basis exploits well the sparsity of the constraint matrix and delivers optimal solutions with high accuracy.
For this reasons, IPX has attracted a lot of attention in the open source community, especially for energy modelling. IPX has been shown to outperform other commonly used open source interior point methods when solving problems coming from energy modelling, and to have a comparable performance to Gurobi \cite{gurobi} on many such instances \cite{openenergy,ParHalJenBro:highs}. 

However, IPX has been shown to struggle to solve certain problems and to have a running time that is difficult to predict beforehand. Due to the complicated nature of the preconditioner and the use of iterative method, it is sometimes difficult to understand the factors that lead to this unfavourable behaviour. Moreover, the approach taken by IPX cannot be generalised to quadratic programming and it does not exploit any parallelism. For these reasons, a new interior point solver has been developed for HiGHS, which uses a more standard factorisation-based approach.

Despite having a higher memory requirement, an interior point method based on a factorisation has many advantages. Its running time can be estimated based on the number of nonzero entries in the factor and on the number of floating point operations, which are known immediately after the symbolic factorisation. The running time is also very consistent, with little variations from one iteration to the other. If implemented properly, a factorisation-based interior point method can achieve a good accuracy on many problems, despite the extreme ill-conditioning of the matrices. However, for many instances, factorising or even just forming the normal equations matrix is not an option, for example due to the presence of dense columns. Being able to work with the alternative approach of the augmented system has many advantages related both to memory constraints and to the numerical properties of the algorithm. However, this approach comes also with challenges, due the indefiniteness of the augmented system matrix. 

Practical implementations of interior point methods have been studied for a long time, see e.g. \cite{AndGonMesXu:implementation,GonTer:computational,Meh:predcor}. The factorisation of the augmented system, combined with the use of regularisation to allow static pivoting, has been analysed e.g. in \cite{AltGon:regularization}. Private communication with developers of commercial solvers confirms that they use similar techniques in their barrier solvers although their developments are not documented in scientific papers.

Indeed, many available implementations either are closed source, with no possibility of knowing the details of the strategies used, or have been developed mainly for research purposes. In the latter case, despite the enormous research value of such implementations, the corresponding codes are not developed for performance and cannot realistically be used for practical or commercial applications, apart from very specific classes of problems. 

To the best of the Authors' knowledge, the solver described in this paper is the first attempt to develop a fully open source interior point method for general purpose linear programming, capable to deliver high performance, that uses the augmented system approach. Many other implementations fail to be fully open source (for example by using a closed source factorisation code), fail to reach high performance, are developed only with specific applications in mind, or use only the normal equations. See e.g. \cite{clp,glpk} for some available implementations.

This paper describes the main features of the new solver {\it HiPO} (HiGHS interior point optimizer) and presents numerical results to compare it with the existing HiGHS interior point solver IPX \cite{SchGon:basisIPM}. Experiments are conducted on the Netlib collection of LPs, on the Mittelmann benchmark, and on a collection of problems coming from the open-source energy models PyPSA-Eur, in order to show the robustness of the new solver and the gains that it can achieve compared to IPX.

We highlight that the new solver has far reaching consequences for some other HiGHS solvers as well: it allows to speed up the solution of mixed integer programs, since it can be used to solve the LP relaxation at the root of the search tree in the near future; it will also speed up the solution of quadratic programs, since the new solver can be easily extended to include a quadratic term.

The rest of the paper is organised as follows: Section~\ref{sec:formulation} introduces the LP formulation used; Section~\ref{sec:ipm} describes the interior point method; Section~\ref{sec:linearsystem} describes the features of the linear solver; Section~\ref{sec:results} presents the numerical results.

\subsection{Notation}
Vectors are represented with bold characters ($\vvec$), vector components as non-bold characters ($v_i$), and matrices with capital letters ($M$). Given a vector $\xvec$, we define a diagonal matrix $X$ such that it has vector $\xvec$ on the diagonal. $I$ represents the identity matrix, $\evec$ the vector of all ones, $\0$ the vector of all zeros, where their size can be deduced from the context. $\|\cdot\|$ represents the Euclidean norm. $\|\vvec\|_\infty$ represents the infinity norm of a vector, i.e., $\max_i |v_i|$.
Given a set $\mathcal S$, we denote its cardinality by $|\mathcal S|$.

\section{Linear programming formulation}
\label{sec:formulation}

Consider a linear program (LP) in the following formulation
\begin{align}
\label{eq:primal}
\min_{\xvec, \xl, \xu} \quad\cvec^T\xvec&\notag\\
\textup{s.t.}\qquad A\xvec &= \bvec\\
\xvec - \xl &= \lvec \notag\\
\xvec + \xu &= \uvec \notag\\
\xl,\xu&\ge\0,\notag
\end{align}
with corresponding dual
\begin{align}
\label{eq:dual}
\max_{\yvec, \zl, \zu} \quad\bvec^T\yvec + \lvec^T\zl - \uvec^T&\zu\notag\\
\textup{s.t.}\qquad\,\, A^T\yvec +\zl - \zu &= \cvec\\
\zl,\zu&\ge\0,\notag
\end{align}
where $A\in\mathbb R^{m\times n}$, $\bvec,\yvec\in\mathbb R^m$, $\cvec,\xvec,\xl,\xu,\lvec,\uvec,\zl,\zu,\in\mathbb R^n$.

A more frequently used formulation to define an LP is the following
\begin{align}
\label{eq:primal_user}
\min_{\tilde\xvec} \quad\tilde\cvec^T\tilde\xvec&\notag\\
\textup{s.t.}\qquad \tilde A\tilde\xvec &\lesseqgtr \tilde\bvec\\
\tilde\lvec\le\tilde\xvec&\le\tilde\uvec,\notag
\end{align}
where each linear constraint can be $\le$, $=$ or $\ge$.

Formulation \eqref{eq:primal_user} is more convenient to be used by a user when inputting a problem in a solver. Formulation \eqref{eq:primal} is more convenient when used internally in an interior point solver. One can easily convert from \eqref{eq:primal_user} to \eqref{eq:primal} simply by adding slack variables to the inequality constraints, and by putting the proper bounds on the slacks. Notice that this formulation is very similar to the one used in \cite{SchGon:basisIPM}.

Define the sets $\mathcal L=\{j:l_j >-\infty\}$ and $\mathcal U=\{j:u_j<+\infty\}$. For ease of exposition, assume that there are no free variables, i.e., $\mathcal L\cup\mathcal U=\{1,\dots,n\}$. Notice that the algorithm can deal with free variables, but they are ignored in the current explanation for simplicity. Notice also that if one of the bounds is infinite, then the corresponding $x^l_j$ or $x^u_j$ is infinite (to guarantee that the linear constraints in \eqref{eq:primal} are satisfied), and the corresponding $z^l_j$ or $z^u_j$ is set to zero (otherwise the objective of \eqref{eq:dual} would be $-\infty$).

\subsection{Scaling}
The original problem provided by the user \eqref{eq:primal_user} is scaled using the Curtis-Reid scaling \cite{CurRei:scaling}. The scaling procedure produces positive diagonal matrices $R$ and $C$ such that $R\tilde AC$ is ``better scaled'' than $\tilde A$. The scaled problem then looks like

\begin{align}
\label{eq:primal_user_scaled}
\min_{C^{-1}\tilde\xvec} \quad (C\tilde\cvec)^T(C^{-1}\tilde\xvec)&\notag\\
\textup{s.t.}\qquad (R\tilde AC)(C^{-1}\tilde\xvec) &\lesseqgtr (R\tilde\bvec)\\
(C^{-1}\tilde\lvec)\le(C^{-1}\tilde\xvec)&\le(C^{-1}\tilde\uvec)\notag.
\end{align}
This scaled problem is then cast into formulation \eqref{eq:primal} by adding slack variables.
Given an approximate solution $(\xvec,\xl,\xu,\yvec,\zl,\zu)$ of the scaled problem, the corresponding unscaled solution is $(C\xvec,C\xl,C\xu,R\yvec,C^{-1}\zl,C^{-1}\zu)$.

Apart from the unscaling process, to recover a solution to the original LP \eqref{eq:primal_user} the signs of the slack variables and Lagrange multipliers for inequality constraints need to be checked carefully. Indeed, there is no explicit sign restrictions on $\xvec$ or $\yvec$ in the internal formulation \eqref{eq:primal}-\eqref{eq:dual}; this means that slacks or Lagrange multipliers may slightly violate the sign constraints that the user expects based on the sign of the inequalities in \eqref{eq:primal_user}.

This issue can be resolved by using variables $\xl,\xu$ to set the correct slacks and $\zl,\zu$ to set the correct Lagrange multipliers. Suppose that, out of $\tilde{n}$ constraints, inequality $i$ requires a slack $x_\iota$ in order to become an equality constraint, where $\iota=\tilde{n}+k$ and $k$ depends on how many slacks have already been added for the previous constraints. Suppose that the inequality has sign $\le$, so that $0\le x_\iota<\infty$. Then, the internal formulation includes a constraint $x_\iota-x^l_\iota=0$, since the variable has a finite lower bound. However, only $x^l_\iota$ is guaranteed to be positive and the constraint could have a small violation. Therefore, the correct slack should be recovered from the value of $x^l_\iota$ rather than $x_\iota$. Similarly, the $\iota^\text{th}$ dual constraint reads $y_i + z^l_\iota = 0$, so that the Lagrange multiplier of inequality $i$ with correct sign should be recovered from $-z^l_\iota$, rather than from $y_i$.

\section{Interior point method}
\label{sec:ipm}

An interior point method (IPM) uses the Newton method to solve the first order optimality conditions of the barrier problem corresponding to \eqref{eq:primal}-\eqref{eq:dual}, see e.g. \cite{Wri:book}. This gives rise to a linear system of the form
\begin{equation}
\label{eq:big_system}
\begin{bmatrix}
A &     &     &     &     &     \\
I &     & -I  &     &     &     \\
I &     &     & I   &     &     \\
  & A^T &     &     & I   & -I  \\
  &     & Z^l &     & X^l &     \\
  &     &     & Z^u &     & X^u \\
\end{bmatrix}
\begin{bmatrix}
\dx \\ \dy \\ \dxl \\ \dxu \\ \dzl \\ \dzu \\
\end{bmatrix}
=
\begin{bmatrix}
\bvec-A\xvec           \\
\lvec-\xvec+\xl        \\
\uvec-\xvec-\xu        \\
\cvec-A^T\yvec-\zl+\zu \\
\sigma\mu\evec-X^lZ^l\evec   \\
\sigma\mu\evec-X^uZ^u\evec   \\
\end{bmatrix}
\eqqcolon
\begin{bmatrix}
\rvec_1 \\ \rvec_2 \\ \rvec_3 \\
\rvec_4 \\ \rvec_5 \\ \rvec_6 \\
\end{bmatrix}
\end{equation}
where $X^l,X^u,Z^l,Z^u$ are diagonal matrices with the corresponding vectors on the diagonal, $\mu$ is the complementarity measure
\[\mu = \frac{1}{|\mathcal L| + |\mathcal U|}\Bigl(\sum_{i\in\mathcal L}x^l_iz^l_i + \sum_{i\in\mathcal U}x^u_iz^u_i\Bigr),
\]
and $\sigma$ is the target reduction in complementarity for the next iteration.

Notice that care must be taken when dealing with the bounds $\lvec$ and $\uvec$, and the corresponding primal and dual variables: if they are infinite, then the corresponding entry in the residual vectors should be ignored and set to zero.

This $(5n+m)\times(5n+m)$ nonsymmetric linear system can be reduced to the {\it augmented system}
\begin{equation}
\label{eq:augsys}
\begin{bmatrix}
-\Theta^{-1} & A^T \\ A & \\
\end{bmatrix}
\begin{bmatrix}
\dx \\ \dy \\
\end{bmatrix}
=
\begin{bmatrix}
\rvec_7 \\ \rvec_1
\end{bmatrix},
\end{equation}
where $\rvec_7=\rvec_4-(X^l)^{-1}(\rvec_5+Z^l\rvec_2)+(X^u)^{-1}(\rvec_6-Z^u\rvec_3)$ and $\Theta$ is a diagonal matrix with 
\[\Theta_{jj}^{-1} = \frac{z^l_j}{x^l_j}+\frac{z^u_j}{x^u_j}.\]
The solution of the original linear system can be recovered as
\begin{align}
\dxl &= \dx-\rvec_2                 \notag\\
\dxu &= \rvec_3 - \dx               \notag\\
\dzl &= (X^l)^{-1}(\rvec_5-Z^l\dxl) \notag\\
\dzu &= (X^u)^{-1}(\rvec_6-Z^u\dxu) \notag.
\end{align}
The linear system \eqref{eq:augsys} can be further reduced to the {\it normal equations} system
\begin{equation}
\label{eq:normeq}
A\Theta A^T\dy = \rvec_1+A\Theta\rvec_7 \eqqcolon \rvec_8
\end{equation}
where $\dx$ is recovered as $\dx=\Theta(A^T\dy-\rvec_7)$.

A typical iteration of an interior point method includes the following steps:
\begin{itemize}
\item Compute the predictor direction, solving \eqref{eq:big_system} with $\sigma=0$.
\item Compute a centrality corrector, aiming at recovering centrality. This involves setting $\rvec_1=\rvec_2=\rvec_3=\rvec_4=\0$ and computing $\rvec_5$ and $\rvec_6$ based on the complementarity products of the point reached with the predictor direction, as described in \cite{ColGon:mcc}.
\item Keep computing correctors until they bring some benefit, or until the maximum number of correctors is reached.
\item Compute the step sizes based on the overall Newton direction and take the step.
\end{itemize}

\subsection{Termination}

Define the primal and dual objective values as $f_p \coloneqq \cvec^T\xvec$ and $f_d\coloneqq \bvec^T\yvec+ \lvec^T\zl - \uvec^T\zu$. The IPM iterations are stopped when the following conditions are met:
\begin{itemize}
\item Primal feasibility: $\|\{\rvec_1; \rvec_2; \rvec_3\}\|_\infty \le \tau_\text{feas} (1+\|\{\bvec;\lvec;\uvec\}\|_\infty)$,
\item Dual feasibility: $\|\rvec_4\|_\infty \le \tau_\text{feas} (1+\|\cvec\|_\infty)$,
\item Optimality: $|f_p - f_d| \le \tau_\text{opt} (1+\frac{1}{2}|f_p+f_d|)$,
\end{itemize}
where $\tau_\text{feas}$ and $\tau_\text{opt}$ are feasibility and optimality tolerances respectively. Notice that the check on optimality is done using the relative primal-dual gap, rather than the value of $\mu$, since the latter depends on the scaling of the problem. Notice also that in practice the primal-dual feasibility termination test is preformed using the quantities coming from the unscaled problem, i.e.,
\begin{itemize}
\item $\|\{R^{-1}\rvec_1; C\rvec_2; C\rvec_3\}\|_\infty \le \tau_\text{feas} (1+\|\{R^{-1}\bvec;C\lvec;C\uvec\}\|_\infty)$,
\item $\|C^{-1}\rvec_4\|_\infty \le \tau_\text{feas} (1+\|C^{-1}\cvec\|_\infty)$,
\end{itemize}
because the process of unscaling the solution may alter the residuals. The primal-dual gap instead is unaffected.

If the user requests crossover, this is performed as described in \cite{SchGon:basisIPM}. In this case, the IPM termination test is stricter than the one stated above, to guarantee that, after dropping variables to their bounds, the primal and dual infeasibilities remain small.

\subsection{Regularisation}
As proposed in \cite{AltGon:regularization}, the linear systems \eqref{eq:augsys} and \eqref{eq:normeq} are slightly perturbed by primal and dual positive diagonal regularisation matrices $R_p$ and $R_d$, so that the matrices of the linear systems become respectively:
\begin{equation}
\label{eq:regul_matrix}
\begin{bmatrix}
-(\Theta^{-1}+R_p) & A^T \\ A & R_d\\
\end{bmatrix},
\qquad
A(\Theta^{-1}+R_p)^{-1} A^T + R_d.
\end{equation}
The regularisation matrices contain both a {\it static} component, i.e., regularisation applied uniformly to all diagonal entries, and a {\it dynamic} component, i.e., regularisation applied only to the diagonal entries which require a more significant perturbation during the factorisation stage, as described later.

The use of regularisation improves the conditioning and spectral properties of the matrices \cite{Gon:matrixfree} and crucially transforms the augmented system from a generic indefinite matrix, to a quasi-definite one. As shown in \cite{Van:quasidefinite}, quasi-definite matrices are strongly factorisable and this gives much needed freedom when choosing the order of pivots during the factorisation stage.

\section{Solving the linear system}
\label{sec:linearsystem}

The Newton direction can be computed by solving either \eqref{eq:augsys} or \eqref{eq:normeq}. This can be achieved with a direct method or an iterative procedure. The current HiGHS interior point method IPX solves the normal equations, using the conjugate residual method and basis preconditioning \cite{SchGon:basisIPM}. This approach allows to solve large problems without having to form the matrix in \eqref{eq:normeq} and without having to compute and store an expensive factorisation; it also allows to obtain very precise Newton directions that guarantee convergence of the IPM for most problems. However, the time required to perform a single IPM iteration varies substantially and is very difficult to predict, because it is impossible to foresee the changes to the basis that the solver selects and the number of Krylov iterations needed to reach the required accuracy. It is also not possible to compute many correctors, since the iterative procedure has the same cost for each of them. Moreover, the approach chosen by IPX does not generalise to quadratic programming and does not exploit multi-threading.

A factorisation-based solver needs to explicitly form the matrix in \eqref{eq:augsys} or \eqref{eq:normeq}, compute its factorisation and store it in memory. This limits the maximum size of the problem that can be solved, depending on the amount of memory available. However, the cost of performing a factorisation is much more stable and predictable. Once the factorisation is available, many solves can be performed with it to compute correctors. Moreover, matrix \eqref{eq:augsys} easily generalises to the quadratic case, with minimal changes to the algorithm \cite{AltGon:regularization}. 

Here and in the rest of the paper, the word {\it solve} is used to refer to the procedure of computing a solution of a linear system using the $LDL^T$ factorisation. This includes a forward solve using $L$, a block-diagonal solve using $D$ and a backward solve using $L^T$.

\subsection{Multifrontal factorisation}
We use a multifrontal factorisation; both the normal equations and augmented system are factorised into $LDL^T$, where $L$ is unit lower triangular and $D$ is block diagonal with blocks of size 1 or 2. The multifrontal factorisation is the most commonly used technique to factorise large sparse matrices, see e.g. \cite{Duf:MA57,DufRei:multifrontal,Liu:multifrontal}.

\begin{figure}
\caption{Example of multifrontal factorisation}
\label{fig:multifrontal}
\centering
\includegraphics[width=.6\textwidth]{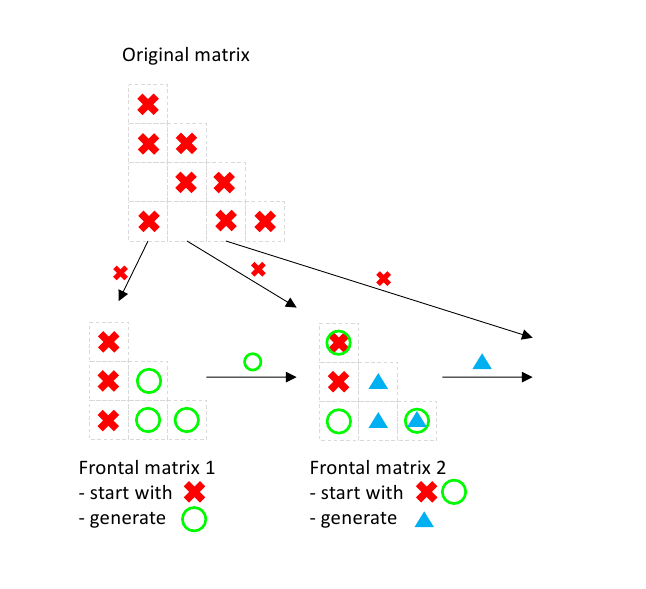}
\end{figure}

Looking at Figure~\ref{fig:multifrontal}, the multifrontal factorisation processes one column at a time. Each column corresponds to a dense {\it frontal matrix} that needs to be factorised, containing the entries of the original matrix and the entries from any Schur complement that affects the current column. For example, the frontal matrix of column 1 contains only the red crosses coming from the original matrix; the Schur complement that it produces after being factorised is shown as green circles. The frontal matrix of column 2 contains the red crosses coming from the original matrix and the green circles coming from the Schur complement of column 1; the second frontal matrix produces as Schur complement the blue triangles, which are then passed on to the frontal matrix of the third column, and so on. See \cite{ScoTum:direct} for a gentle introduction to the topic of direct solvers.

The Schur complements are stored only for the time needed to pass them from one frontal matrix to another. The first column of each frontal matrix instead becomes a column of the factor $L$. Modern linear solvers do not work on individual columns, but rather on supernodes, i.e., blocks of consecutive columns with the same sparsity structure in the factor $L$ \cite{LiuNgPey:supernodes}. This is done to achieve better efficiency on modern computer architectures and to reduce the number of operations needed to pass the Schur complements from one frontal matrix to the other.

From the above description, it should be clear that the key to an efficient sparse factorisation routine lies in finding dense substructures within the sparse matrix. Then, two types of operations need to be performed:
\begin{itemize}
\item Dense operations, e.g. the dense factorisation of the frontal matrices. These exploit efficient dense routines like BLAS \cite{DonDucHamDuf:blas} and achieve a very high peak of CPU performance. Such operations can also be parallelised efficiently. 
\item Sparse operations, e.g. summing the contribution of one Schur complement into another frontal matrix. These operations involve more communication and irregular memory accesses and therefore can be quite slow. Their impact can be reduced by amalgamating supernodes \cite{AshGri:supernodes}.
\end{itemize}

\subsection{Ordering and symbolic factorisation}

Before performing the actual factorisation, the linear solver performs a nested dissection ordering using Metis \cite{KarKum:metis}. On top of being a good quality fill-reducing heuristic, nested dissection ordering produces an elimination tree that is ``wide'' rather than ``tall'', i.e., it has more independent branches and is more suitable to be parallelised.

The IPM needs to solve linear systems that always have the same sparsity structure. Therefore, the ordering and symbolic factorisation need only be computed once and can be reused for all subsequent factorisations. This is an important advantage, since running Metis on large problems can be expensive. 

In principle, ordering and symbolic factorisation of both the normal equations \eqref{eq:normeq} and augmented system \eqref{eq:augsys} can be performed and then the better one can be selected. However, for some problems, forming the normal equations matrix alone can be much more expensive than factorising the augmented system. Therefore some precautions are required to avoid potentially expensive and unnecessary operations. For example, the symbolic factorisation of the augmented system should be run first; if, while forming the normal equations, the number of nonzero entries in the matrix is larger than the number of nonzero entries in the factorisation of the augmented system, the normal equations approach can be discarded before even terminating the forming of the matrix. Naturally, one approach may be preferred to the other for various reasons, for example due to the presence of dense columns.

\subsection{Elimination tree}
A fundamental tool to implement a multifrontal solver is the {\it elimination tree}, which gives the dependencies between columns (or supernodes), i.e., the list of columns that need to be factorised before a given column can be processed. Figure~\ref{fig:tree} shows an example of an elimination tree: column 3 needs to wait until columns 1 and 2 are processed, before being factorised.
\begin{figure}
\caption{Example of elimination tree}
\label{fig:tree}
\centering
\includegraphics[width=.6\textwidth]{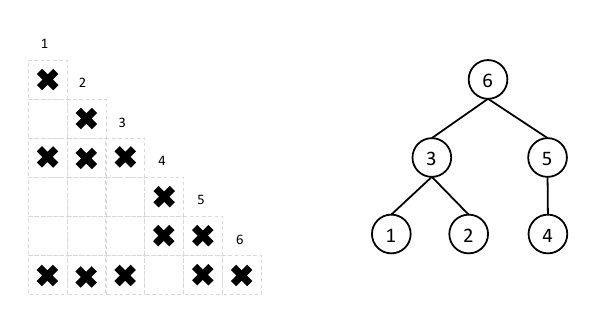}
\end{figure}
The shape of the tree determines the branches that can be processed in parallel. In Figure~\ref{fig:tree}, the branches 1-2-3 and 4-5 can be processed independently, on different threads; node 6 needs to wait for both these branches to complete.

Notice that the elimination tree can be computed efficiently \cite{Dav:direct}. Observe also that the tree that is actually used in practice is the supernodal elimination tree, which gives dependencies among supernodes, rather than single columns.

The issue naturally arises of when to allocate memory for a given  frontal matrix, since the multiple Schur complements to sum into it can become available at different stages of the factorisation. Much research has gone into this topic which is too complicated to be described here. Instead, we refer the interested reader to \cite{GueLex:multifrontal}.

When dealing with a parallel code, it is important to obtain {\it deterministic} behaviour (also known as {\it bit compatibility}), i.e., multiple runs of the code with the same input should always return the exact same output and the exact same sequence of intermediate states. This is important for reproducibility of the results and to facilitate debugging. To achieve determinism, the Schur complements need to be summed into the frontal matrix always in the same order. In Figure~\ref{fig:tree}, nodes 3 and 5 may be processed in parallel, while node 6 waits for them; however, node 6 should decide on a deterministic order in which to sum the contributions from nodes 3 and 5. The fact that node 5 may be ready to be summed before node 3, or vice versa, should not change the order of operations performed, since this can lead to non deterministic results, due to the properties of floating point arithmetic. See \cite{HogSco:report} for more details.

The parallel processing of the elimination tree just described is know as {\it tree-level} parallelism. A second type of parallelism is available in a multifrontal solver, known as {\it node-level} parallelism: this is the act of exploiting multiple threads to perform the dense factorisation of a specific frontal matrix. This strategy is particularly effective when dealing with nodes close to the root of the tree, since these are likely to have much larger frontal matrices, and it is likely that very few branches are left to be processed in parallel, reducing the contention for available threads. Notice that determinism is a concern for this strategy too: the main operations that can be parallelised are the large matrix-matrix products needed to update the Schur complement in the blocked dense factorisation, and these can be performed easily in a deterministic way.

\subsection{Dense matrices}

An important aspect for the efficiency of the code is the format used to store the dense frontal matrices. There are two objectives: minimising the amount of memory needed, and guaranteeing good cache locality when calling BLAS functions. Since the matrices involved are symmetric, there is no need to store the upper triangular part of the matrix, as in Figure~\ref{fig:format1}. However, doing so means that the rows of the matrix are not well aligned in memory. To improve performance, dense matrices are stored in blocks of columns, in which each block stores explicitly the corresponding upper triangular part, in order to maintain advantageous alignment of the data. Within each block of columns, the data can be stored by columns, as in Figure~\ref{fig:format2}, or by rows, as in Figure~\ref{fig:format3}. As it was shown in \cite{AndGunGusReiWas:format}, using a format that stores the blocks of columns by rows guarantees a better cache locality when performing the dense factorisation of the frontal matrices. Therefore, the format shown in Figure~\ref{fig:format3} is used during the factorisation of the frontal matrices. However, this format is impractical to use while assembling the contributions of the Schur complements into a frontal matrix; it is more natural to assemble the contributions by columns. So, the matrix is only put into the hybrid format of Figure~\ref{fig:format3} to perform the dense factorisation and is kept in the format of Figure~\ref{fig:format1} for the remaining operations. Notice that when using complicated pivoting strategies, it may be advantageous to use the format of Figure~\ref{fig:format2}, since it gives easier access to the full pivotal column.

\begin{figure}
\caption{Different storage formats for dense matrices: white squares are not stored, light blue squares are stored but not used, dark blue squares are stored and used. The numbers indicate the address of the entry within the array.}
\centering
\begin{subfigure}{0.28\textwidth}
\centering
\includegraphics[width=\textwidth]{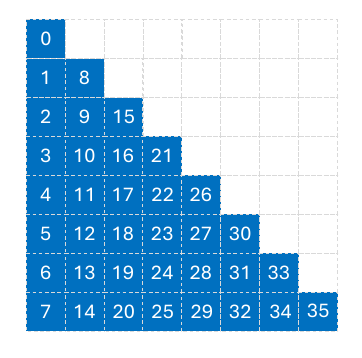}
\caption{}
\label{fig:format1}
\end{subfigure}
\begin{subfigure}{0.28\textwidth}
\centering
\includegraphics[width=\textwidth]{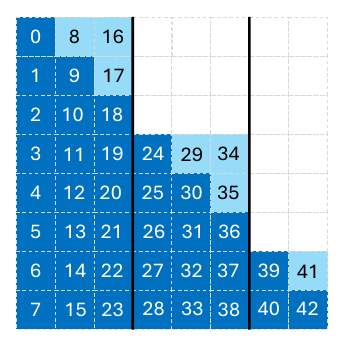}
\caption{}
\label{fig:format2}
\end{subfigure}
\begin{subfigure}{0.28\textwidth}
\centering
\includegraphics[width=\textwidth]{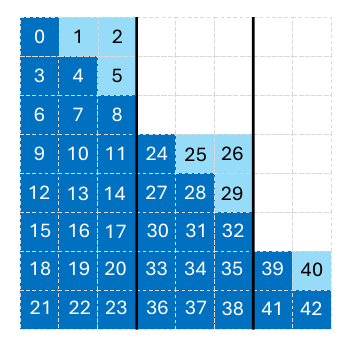}
\caption{}
\label{fig:format3}
\end{subfigure}
\end{figure}

The size of the blocks should be chosen depending on the amount of cache memory available. Common sizes for modern machines are 32, 64 or 128 columns.

Looking at Figure~\ref{fig:blocks}, we can identify multiple blocks of columns within a given frontal matrix: the block $B_D$ is the diagonal block currently being factorised; all the pivots before block $B_D$ have already been eliminated. Block $B_P$ consists of the rest of the pivotal columns corresponding to pivots in $B_D$. Blocks $B_U$ and $B_V$ have already been computed and are used to update the current pivotal columns. $D_V$ indicates the portion of matrix $D$ that includes the pivots corresponding to columns in blocks $B_V$ and $B_U$.

The dense factorisation proceeds as a sequence of BLAS calls: 
\begin{itemize}
\item level-3 subroutine \texttt{dgemm} is used to compute $B_D\leftarrow B_D-B_V D_V B_V^T$ and $B_P\leftarrow B_P-B_U D_V B_V^T$, for each block to the left of $B_D$.
\item level-3 subroutine \texttt{dtrsm} is used to compute $B_P\leftarrow B_P L_{BD}^{-T}$, where $L_{BD}$ is the triangular factor of $B_D$.
\item level-2 subroutines are used to factorise $B_D$ using a {\it factorisation kernel}.
\item level-2 subroutines \texttt{dgemv} and \texttt{dtrsv} are used in the solve phase to perform matrix-vector products and solves with triangular matrices.
\item many level-1 subroutines are used throughout the code to copy, add, scale and swap vectors.
\end{itemize}
See \cite{DonDucHamDuf:blas} for a description of BLAS functions.

\begin{figure}
\caption{Blocks for dense factorisation.}
\label{fig:blocks}
\centering
\includegraphics[width=.45\textwidth]{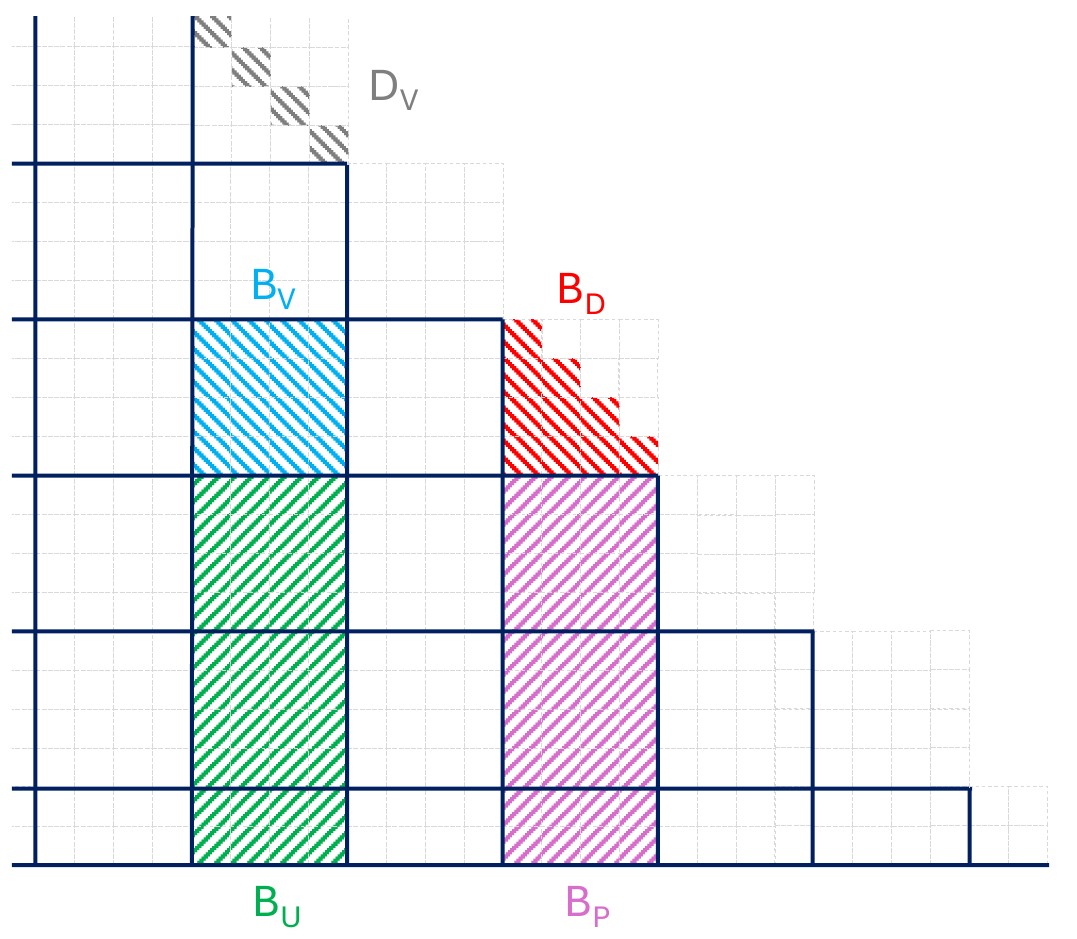}
\end{figure}

\subsection{Pivoting and regularisation}
\label{sec:stability}

When performing matrix factorisation, it is important to avoid selecting small pivot elements; otherwise, very large entries can appear in matrix $L$, leading to an accumulation of round-off errors and eventual loss of accuracy during the rest of the factorisation and the solve phase. We employ the following two techniques to prevent this form happening:
\begin{itemize}
\item Pivoting: pivots are only accepted if they are large enough, otherwise columns are interchanged to obtain a valid pivot, or a $2\times2$ pivot is used. This can lead to the phenomenon of delayed pivots, which increases in an unpredictable way the amount of time and memory required to complete the factorisation. See e.g. \cite{AshGriLew:pivoting}.
\item Perturbation: pivots that are not acceptable are regularised (perturbed) so that their magnitude becomes large enough. This technique produces a factorisation of a modified matrix, but it has the advantage of avoiding expensive pivoting operations. See e.g. \cite{AltGon:regularization}.
\end{itemize}

If the pivoting operations are performed only within a certain supernode, then the same sparsity pattern determined by the symbolic factorisation can be used and no changes to the data structures are needed. This is known as {\it static pivoting}. If however it is not possible to choose a valid pivot within the supernode, then perturbation must be applied to deliver an acceptable pivot and complete the factorisation successfully.

In the context of interior point methods, the perturbation can be interpreted as a regularisation of the underlying optimization problem, which has been studied extensively, see e.g. \cite{AltGon:regularization,CipGon:psipm,FriOrb:regularization}. It is thus less problematic to work with a modified matrix, as long as the perturbation remains small enough. Notice that, when using the augmented system approach, the pivots coming from the $(1,1)-$block in \eqref{eq:augsys} should receive a negative contribution from the regularisation, while the others should receive a positive contribution, in order to preserve the correct sign of the pivots and the correct inertia of the matrix. 

Based on the ideas of \cite{SchGar:pivoting}, we use the following strategy, which is a variation of Bunch-Kaufman pivoting \cite{BunKau:pivoting}. Given a matrix $M$ for which $p-1$ pivots have already been eliminated, the next pivot candidate is $M_{pp}$. Define $r$ so that $M_{rp}$ is the entry of largest magnitude in column $p$; call this magnitude $\gamma_p\coloneqq|M_{rp}|$. Find the entry of largest magnitude in row/column $r$ and call this magnitude $\gamma_r$.
\begin{itemize}
\item If $\max(|M_{pp}|,|M_{rp}|)\le \delta \|M\|_1$, then lift the pivot so that it has magnitude $\delta\|M\|_1$ and select it.
\item Otherwise, if $|M_{pp}|\ge\alpha\gamma_p$, then select pivot $p$.
\item Otherwise, if $|M_{pp}|\ge\alpha\gamma_p^2/\gamma_r$, then select pivot $p$.
\item Otherwise, if $|M_{rr}|\ge\alpha\gamma_r$, then select pivot $r$.
\item Otherwise, select the $2\times2$ pivot $(p,r)$.
\end{itemize}
$\alpha$ is a parameter that is often set to $(\sqrt{17}+1)/8$, but other values can be used as well. $\delta$ is a parameter that controls the amount of perturbation introduced and is of the order of magnitude of machine precision.

Notice that the procedure described above requires swapping columns of the matrix and potentially accepting $2\times2$ pivots, so that matrix $D$ is no longer diagonal, but only block diagonal. Notice also that this procedure only applies to the factorisation kernel used to factorise the diagonal block ($B_D$ in Figure~\ref{fig:blocks}); this means that only the portion of pivotal columns within $B_D$ are considered, while the entries in $B_P$ are not used during pivoting.

We found that it is beneficial to perform an additional swap of columns before the pivoting procedure commences, to put the pivot with largest magnitude in position $M_{pp}$. This guarantees that if a strong candidate is present, then it will be selected. The extra swaps that are potentially added are balanced by the extra accuracy that can be obtained with a good pivot choice. 

Despite knowing that Bunch-Kaufman pivoting is not a robust pivoting strategy in general \cite{AshGriLew:pivoting}, we found that this technique, combined with the dynamic regularisation described later, is sufficiently robust for the vast majority of problems, while having a low cost. More complicated pivoting strategies would require more expensive operations, since they need to have access to the whole pivotal column rather than just the portion within the diagonal block $B_D$.

Looking at \eqref{eq:regul_matrix}, the perturbation performed during pivoting corresponds to the dynamical regularisation part of matrices $R_p$ and $R_d$. In order to improve the numerical properties of the highly ill-conditioned IPM matrices, a static regularisation term is also applied to all pivots regardless of their magnitude. This perturbation is in the range $10^{-10}-10^{-12}$, and it is applied right after the pivot has been computed, rather than when the matrix is formed. In this way, the effect of numerical cancellation is reduced. Take for example a diagonal entry of magnitude $10^{10}$ in the original matrix; if static regularisation is applied at the beginning, then it would completely disappear. However, if the diagonal entry turns out to become a very small pivot, say of order $10^{-14}$, then the addition of static regularisation would have been very beneficial. By applying it after the pivot has been computed, we are sure to obtain the desired effect.

Despite the pivoting procedure and the perturbation, the matrices involved are so ill conditioned that occasionally an iterative refinement procedure still needs to be applied to obtain a good quality solution.

Suppose that the linear system to be solved is $M\dvec=\fvec$, which can correspond to either \eqref{eq:augsys} or \eqref{eq:normeq}. After computing $\dvec = M^{-1}\fvec$ using the $LDL^T$ factorisation, the following iterative procedure is repeated: 
\begin{itemize}
\item Compute the residual $\rvec=\fvec-M\dvec$.
\item Compute the correction $\tilde\dvec = M^{-1}\rvec$.
\item Update $\dvec\leftarrow\dvec+\tilde\dvec$.
\end{itemize}
This process is repeated until the residual is small enough, or until the maximum number of iterations is reached (usually 1 or 2). To evaluate how small the residual is, we use the strategy described in \cite{AriDemDuf:error}, which computes the component-wise backward error
\[\frac{|M\dvec-\fvec|_i}{\Large(|M||\dvec|+|\fvec|\Large)_i}\]
and substitutes small denominators with $\Large(|M||\dvec|+\|\dvec\|_\infty|M|\evec\Large)_i$, where the absolute value of a matrix or vector corresponds to the same matrix or vector with all entries taken in absolute value.

\subsection{Quasi-definiteness and signs of the pivots}
Adding the regularisation to the augmented system as in \eqref{eq:regul_matrix} makes the matrix quasi-definite, since the $(1,1)-$block is strictly negative definite and the $(2,2)-$block is strictly positive definite. Quasi-definiteness implies that the matrix is strongly factorisable \cite{Van:quasidefinite}, i.e., any pivot order produces a stable factorisation (assuming that exact arithmetic is used). Equivalently, after elimination of any pivot, the remaining Schur complement still remains quasi-definite. For completeness, we report a proof of the last statement below.

The presence of round-off errors may however affect the quasi-definiteness of the matrix and produce a numerically unstable factorisation. Whenever a pivot is regularised, as part of the pivoting strategy described in the previous section, its value must be large enough to guarantee that the factorisation is sufficiently stable (this can be interpreted as choosing the value $\delta$ of the pivoting strategy). This section describes some crucial aspects to consider when performing regularisation, and a simple yet powerful heuristic that can be used in practice.

Consider a quasi-definite matrix during the process of elimination. At some stage, pivot $p$ needs to be eliminated. Without loss of generality, suppose that the remaining rows and columns are permuted so that the following assumption holds:

\begin{assumption}
\label{assum:structure}
The matrix has the following structure
\[
\begin{bmatrix}
p & \qvec_1^T & \qvec_2^T \\
\qvec_1 & M_{11} & M_{21}^T \\
\qvec_2 & M_{21} & M_{22}
\end{bmatrix}
\]
where $p<0$, $\qvec_1\in\mathbb R^{n_1}$, $\qvec_2\in\mathbb R^{n_2}$, $M_{11}\in\mathbb R^{n_1\times n_1}$, $M_{21}\in\mathbb R^{n_2\times n_1}$, $M_{22}\in\mathbb R^{n_2\times n_2}$, $M_{11}\prec 0$, $M_{22}\succ0$.
\end{assumption}

To guarantee that the matrix is quasi-definite, we assume also the following:
\begin{assumption}
\label{assum:quasi}
\[\begin{bmatrix}
p & \qvec_1^T \\ \qvec_1 & M_{11}
\end{bmatrix}\prec0.
\]
\end{assumption}

We consider only the case of a negative pivot, similar arguments can be used when the pivot $p>0$ originates from the $(2,2)-$block.

\begin{proposition}
If Assumptions \ref{assum:structure} and \ref{assum:quasi} hold, then, after eliminating pivot $p$, the remaining Schur complement is quasi-definite.
\end{proposition}
\begin{proof}
Eliminating pivot $p$ leaves a matrix of size $n_1+n_2$ to be processed in the next steps. The $(1,1)-$block of the remaining matrix is $M_{11}-\frac{1}{p}\qvec_1\qvec_1^T$; the $(2,2)-$ block reads $M_{22}-\frac{1}{p}\qvec_2\qvec_2^T$. Since $p<0$ and $M_{22}\succ0$, the $(2,2)-$block is strictly positive definite. 

Consider now Assumption~\ref{assum:quasi}, i.e.,
\[
\begin{bmatrix}
\alpha & \xvec^T
\end{bmatrix}
\begin{bmatrix}
p & \qvec_1^T \\ \qvec_1 & M_{11}
\end{bmatrix}
\begin{bmatrix}
\alpha \\ \xvec
\end{bmatrix}
<0
\]
for any $\alpha\in\mathbb R$, $\xvec\in\mathbb R^{n_1}$. This is equivalent to 
\begin{align}
0>\,\,& \alpha^2p+2\alpha \qvec_1^T\xvec +\xvec^TM_{11}\xvec \notag\\
=\,\,& \xvec^TM_{11}\xvec-\frac{1}{p}(\qvec_1^T\xvec)^2+\frac{1}{p}\Big[(\qvec_1^T\xvec)^2+2\alpha p \qvec_1^T\xvec+\alpha^2 p^2\Big] \notag\\
=\,\,& \xvec^T\Big(M_{11}-\frac{1}{p}\qvec_1\qvec_1^T\Big)\xvec+\frac{1}{p}\Big(\qvec_1^T\xvec+\alpha p\Big)^2 \notag.
\end{align}
Choosing $\alpha=-\qvec_1^T\xvec/p$ yields that $M_{11}-\frac{1}{p}\qvec_1\qvec_1^T$ is negative definite. Therefore, after eliminating pivot $p$, the remaining Schur complement has a negative definite $(1,1)-$block and a positive definite $(2,2)-$block.
\end{proof}

Therefore, the elimination of a pivot from a quasi-definite matrix leaves behind another quasi-definite matrix. This guarantees that the whole factorisation can proceed without breaking down, regardless of the order of the pivots. It also implies that the diagonal entries of the matrix maintain always their original sign. See \cite{Van:quasidefinite} for a more complete analysis of the factorisation of quasi-definite matrices.

Suppose now that, due to round-off errors, the pivot that is actually used in practice is $\tilde p$, which is still negative, but of smaller magnitude than the correct pivot $\bar p$. We can no longer be sure that the matrix at the current stage is quasi-definite. However, we assume that, if the correct pivot $\bar p$ was used, the matrix would be quasi-definite, i.e., $0>\alpha^2\bar p+2\alpha \qvec_1^T\xvec +\xvec^TM_{11}\xvec$ holds for any $\alpha$ and $\xvec\in\mathbb R^{n_1}$.

The proposition below sheds light on potential consequences of the use of perturbed pivot $\tilde p$ (which may have 
accumulated round-off errors) instead of the original one $\bar p$. We assume that the pivot originates from the negative definite part of matrix hence $\bar p < 0$ and the perturbed value of it $\tilde p$ may get dangerously close to zero.

\begin{proposition}
Given the structure in Assumption~\ref{assum:structure} and assuming that Assumption~\ref{assum:quasi} holds with pivot $\bar p$, given a perturbed pivot $\tilde p$ such that $\bar p<\tilde p<0$, then the following holds $\forall\xvec\in\mathbb R^{n_1}$
\[\xvec^T\Big(M_{11}-\frac{1}{\tilde p}\qvec_1\qvec_1^T\Big)\xvec
+\frac{\bar p-\tilde p}{\bar p\tilde p}(\qvec_1^T\xvec)^2 < 0.\]
\end{proposition}

\begin{proof}
We start by the assumption that the matrix would be quasi-definite, if the unperturbed pivot was used. Then:
\begin{align}
0>\,\,& \alpha^2\bar p+2\alpha \qvec_1^T\xvec +\xvec^TM_{11}\xvec \notag\\
=\,\,& \xvec^TM_{11}\xvec-\frac{1}{\bar p}(\qvec_1^T\xvec)^2+\frac{1}{\bar p}\Big[(\qvec_1^T\xvec)^2+2\alpha \bar p \qvec_1^T\xvec+\alpha^2 \bar p^2\Big] \notag\\
=\,\,& \xvec^TM_{11}\xvec
-\frac{1}{\tilde p}(\qvec_1^T\xvec)^2
+\frac{1}{\tilde p}(\qvec_1^T\xvec)^2
-\frac{1}{\bar p}(\qvec_1^T\xvec)^2
+\frac{1}{\bar p}\Big[(\qvec_1^T\xvec)^2+2\alpha \bar p \qvec_1^T\xvec+\alpha^2 \bar p^2\Big] \notag\\
=\,\,& \xvec^T\Big(M_{11}-\frac{1}{\tilde p}\qvec_1\qvec_1^T\Big)\xvec
+(\qvec_1^T\xvec)^2\frac{\bar p-\tilde p}{\bar p\tilde p}
+\frac{1}{\bar p}\Big(\qvec_1^T\xvec+\alpha \bar p\Big)^2 \notag.
\end{align}
Choosing $\alpha=-\qvec_1^T\xvec/\bar p$ we obtain the thesis.
\end{proof}
Given that $\bar p$ and $\tilde p$ have the same sign and $\bar p<\tilde p$, the quantity $\frac{\bar p-\tilde p}{\bar p\tilde p}$ is strictly negative. Therefore, the Schur complement remaining after eliminating pivot $\tilde p$ may or may not be negative definite. If $\tilde p$ becomes arbitrarily close to zero, then $\frac{\bar p-\tilde p}{\bar p\tilde p}$ becomes arbitrarily large in magnitude; in turn, this means that the quantity $\xvec^T\big(M_{11}-\frac{1}{\tilde p}\qvec_1\qvec_1^T\big)\xvec$ can potentially attain arbitrarily large positive values, without violating the inequality. By allowing a pivot that is too close to zero, we lose control of the state of the matrix after the elimination of such pivot. 

If instead the pivot changes sign and becomes positive, then the $(1,1)-$block of the Schur complement is guaranteed to be negative definite, but the $(2,2)-$block may not be positive definite any more. This means that the issue is not contained within the negative definite part of the matrix any more, but it has spread to the positive definite part as well. Allowing a pivot to change sign can therefore amplify the issues related to small pivots. We therefore consider it crucial to prevent round-off errors from switching the signs of pivots. 

It is particularly difficult to assess whether, after the elimination of a given pivot, the remaining Schur complement is still quasi-definite. It is much simpler to determine whether the elimination of a given pivot is going to switch the signs of the diagonal entries of the Schur complement.
Despite the latter condition being weaker than the former, we argue that monitoring only the signs of the diagonal entries of the Schur complement produces a factorisation that is stable enough to be used within the interior point iterations. The next proposition formalises this heuristic.

\begin{proposition}
\label{prop:regularise}
Given the structure in Assumption~\ref{assum:structure}, if $p<0$ and
\[
|p|\ge\max_j \frac{(\qvec_1)_j^2}{-(M_{11})_{jj}}
\]
then the diagonal entries of the Schur complement after eliminating pivot $p$ have the correct signs.
\end{proposition}
\begin{proof}
Since $p$ is negative, the $(2,2)-$block of the Schur complement is positive definite and therefore its diagonal entries are all positive. A generic diagonal entry of the $(1,1)-$block of the remaining matrix reads
\[(M_{11})_{jj} - \frac{1}{p}(\qvec_1)_j^2.\]

To maintain the correct sign, these entries need to be negative after the elimination. Considering that $p<0$, this is equivalent to
\[(M_{11})_{jj} + \frac{1}{|p|}(\qvec_1)_j^2<0, \,\,\forall j.\]
Rearranging the terms gives the thesis.
\end{proof}
Therefore, if the pivot is regularised so that its magnitude is not smaller than the bound given in Proposition~\ref{prop:regularise}, then all the diagonal entries of the Schur complement have the correct sign. This does not necessarily guarantee that the remaining matrix is quasi-definite, but it gives some minimum reassurance that the elimination process will continue without breaking down.

If the regularisation required according to Proposition~\ref{prop:regularise} is too large, this may be an indication that the factorisation process went catastrophically wrong. In such cases, it may be beneficial to restart the whole factorisation adding a larger static regularisation, or even discard the interior point iteration altogether and perform different algorithmic choices, for example to obtain a better centred iterate.

Notice that, in practice, this regularisation technique can be applied only within the diagonal block $D_B$, since the full pivotal column is not known during the factorisation of the block. This means that some diagonal entries not included in $D_B$ may still change sign. Other heuristics can be applied to detect and correct this behaviour. 

Notice also that many strategies could be applied to prevent pivots from changing signs, for example full pivoting. However, such techniques would be much more expensive and would require changes to the data structures needed for the factorisation. The heuristic presented here requires only few extra operations to be computed each time a pivot is selected. Private communication with developers of commercial solvers confirmed that full pivoting is usually not applied in the context of interior point methods; regularisation strategies and practical heuristics like the ones described in this section are employed instead.

\subsection{Cost of solves}
\label{sec:correctors}

Performing iterative refinement requires more solves, since each correction requires a solution with the factorisation $LDL^T$. Although the cost of solves is usually smaller than the cost of the factorisation, performing too many solves per factorisation can become expensive. For this reason, the number of correctors used during the IPM should be chosen based on the properties of the problem. For a given number of correctors $k$, there are $1+k$ directions that need to be computed, the predictor and $k$ correctors. Each of these directions requires the initial solve and up to $f$ refinement steps (typical values are $f=1$ or $2$).
 Since the refinement often stops before reaching the maximum number of steps, the number of solves per IPM iteration can be estimated to be roughly $(1+k)(1+f/2)$. It is not possible to know how much improvement a single corrector brings to the optimization solver; a simple heuristic to choose $k$ is to balance the effort needed for the factorisation and for the solves, i.e., to guarantee that the time spent doing solves is not larger than the time spent doing the factorisation. 
 The effort for a single solve $E_s$ is proportional to two times the number of nonzero entries in $L$, since each solve requires a forward and backward pass through $L$. The effort for the factorisation $E_f$ depends on the number of floating point operations needed (this is computed during the symbolic analysis) and on the number of sparse indexing operations. The relative cost of these two operations need to be estimated empirically, as it depends on many variables and different features of the computer architecture used. The factorisation effort should also be scaled by a coefficient $\alpha_F$, to take into account the higher efficiency of this phase compared to solve, due to parallelism and the use of higher level BLAS. To summarise, we can estimate the number of correctors $k$ using
 \[(1+k)(1+f/2) E_s \approx E_f  \alpha_F. \]
The resulting number $k$ should then be lifted to be at least $1$, since the first corrector is always beneficial. Additionally, a maximum number of correctors may be set.

Unfortunately, the multiple centrality correctors \cite{ColGon:mcc} are serial in nature, since the right-hand side for a corrector depends on the point reached with the previous one. There is no easy way to overcome this drawback and compute equally successful correctors which could benefit from parallelism.



\section{Results}
\label{sec:results}

We present computational results of applying HiPO to solve problems from three test sets: 
\begin{itemize}
\item The Netlib collection \cite{netlib} consists of 98 linear programs of small to medium size. These problems are commonly used to assess the robustness of LP solvers.
\item The Mittelmann collection \cite{mittelmann} consists of 49 linear programs of medium to large size. These problems are commonly used to benchmark commercial and open source solvers.
\item A collection of 45 medium size linear programs coming from energy systems modelling using PyPSA-Eur \cite{HorHofSchBro:pypsa}. These problems are particularly interesting, since much of the interest in HiGHS comes from the open-source energy modelling community. These instances can be found at \cite{openenergy}.
\end{itemize}

The experiments are run on a 3.2GHz AMD EPYC 7262 processor with 8 cores, 16 threads, 128GB of memory, running Linux.
The code is written in \cpp 11, built using CMake and compiled with Clang 14.0. The BLAS library used is OpenBLAS 0.3.29. Shared-memory parallelism is obtained using the HiGHS parallel scheduler, which exploits the \cpp multi-threading environment, and uses half of the available threads, 8 in this case (this is default behaviour of HiGHS to avoid excessive loads on the machine). OpenBLAS is run in serial mode, since running it in parallel produces non-deterministic results, and since the parallelisation of matrix-matrix products is already performed as part of the node-level parallelism.

All the problems are presolved using HiGHS before being passed to HiPO. Appendix~\ref{apx:size} shows the size of the Mittelmann and Pypsa problems, before and after presolve, in terms of rows, columns and number of nonzero entries of the constraint matrix. We do not report the size of the Netlib problems, but they can be found in \cite{netlib}.
All problems are solved with accuracy $\tau_\text{feas} = \tau_\text{opt} = 10^{-8}$. Only the time taken by the interior point solver is reported, without including the time to read and presolve the problem. While these times can be large sometimes, they are exactly the same for the two solvers that are compared, and they are not the subject of interest of this paper.

If the solver fails to deliver the required accuracy using the direct factorisation, then the optimization can be continued using the basis preconditioner of IPX, which usually yields very accurate Newton directions. If crossover is needed, then IPX can be used to perform that as well. For the sake of space, only results without crossover are shown; results on the crossover of IPX can be found in \cite{SchGon:basisIPM}.

Notice that each problem uses a different number of correctors, as described in Section~\ref{sec:correctors}. Therefore, the number of IPM iterations can vary substantially among different problems.

As mentioned before, the code is deterministic, so multiple runs with the same input produce the same output. Changing the number of threads also does not affect the output of the solver; however, running the code with the exact same configuration but on different machines can lead to different results, due to the potentially different optimizations performed by the compiler. Linking with a different BLAS library also has a strong impact on the results that are produced.

The code is available under MIT license as part of HiGHS at \url{github.com/ERGO-Code/HiGHS}; to be built successfully, it requires Metis 5 from \url{github.com/KarypisLab/METIS}, GKlib from \url{github.com/KarypisLab/GKlib}, and a BLAS library.

\subsection{Netlib results}
The results of HiPO performance on the Netlib collection are reported in Table~\ref{tab:netlib_results}.

Notice that the column of IPM iterations reports in parentheses the number of iterations that were performed at the end using the basis preconditioner to guarantee reaching the required accuracy. 
For most of these problems, the normal equations approach is selected. The exceptions are \texttt{fit1p}, \texttt{fit2p}, \texttt{israel}, \texttt{modszk1}, \texttt{sc205}, \texttt{stocfor2}, \texttt{stocfor3}. This choice is made automatically using a heuristic technique.

The problems from Netlib are rather small for today's computers and HiPO does not display any significant advantage over IPX. Notice that the whole collection of 98 problems can be read from file, presolved and solved to optimality in under 40 seconds. Therefore, these results should be considered only as evidence that the solver is robust. The other two collections of problems are more challenging and the two solvers compared display noticeably different performance.

\begin{table}
\tiny
\caption{Results on the 98 Netlib problems. The column ``Iter'' shows also in parentheses the number of IPX iterations performed at the end (if any).}
\label{tab:netlib_results}
\centering
\begin{tabular}{l|rr||l|rr||l|rr}
\toprule
Problem & Iter & Time & Problem & Iter & Time & Problem & Iter & Time \\
\midrule
25fv47 & 29 & 0.124 & ganges & 14 & 0.024 & scfxm3 & 29 & 0.084 \\
80bau3b & 38 & 0.347 & gfrd & 14 & 0.016 & scorpion & 13 & 0.003 \\
adlittle & 12 & 0.003 & greenbea & 31 & 0.403 & scrs8 & 19 & 0.014 \\
afiro & 8 & 0.000 & greenbeb & 34 & 0.250 & scsd1 & 10 & 0.007 \\
agg & 19 & 0.013 & grow15 & 20 & 0.029 & scsd6 & 13 & 0.015 \\
agg2 & 19 & 0.028 & grow22 & 22 & 0.051 & scsd8 & 12 & 0.031 \\
agg3 & 19 & 0.028 & grow7 & 17 & 0.012 & sctap1 & 20 & 0.020 \\
bandm & 18 & 0.015 & israel & 23 & 0.037 & sctap2 & 18 & 0.070 \\
beaconfd & 9 & 0.000 & kb2 & 13 & 0.002 & sctap3 & 19 & 0.102 \\
blend & 10 & 0.002 & lotfi & 18 & 0.008 & seba & 8 & 0.000 \\
bnl1 & 41 & 0.093 & maros-r7 & 68 & 5.541 & share1b & 20 & 0.007 \\
bnl2 & 29 & 0.182 & maros & 22 & 0.050 & share2b & 12 & 0.004 \\
boeing1 & 24 & 0.035 & modszk1 & 23(3) & 0.153 & shell & 21 & 0.023 \\
boeing2 & 16 & 0.010 & nesm & 31 & 0.103 & ship04l & 16 & 0.026 \\
bore3d & 14 & 0.002 & perold & 28 & 0.090 & ship04s & 21 & 0.021 \\
brandy & 21 & 0.010 & pilot.ja & 29 & 0.181 & ship08l & 16 & 0.043 \\
capri & 26 & 0.022 & pilot & 34 & 0.738 & ship08s & 15 & 0.021 \\
cycle & 26 & 0.139 & pilot.we & 31 & 0.098 & ship12l & 15 & 0.054 \\
czprob & 23 & 0.046 & pilot4 & 31 & 0.065 & ship12s & 15 & 0.023 \\
d2q06c & 30 & 0.368 & pilot87 & 43 & 2.169 & sierra & 17 & 0.058 \\
d6cube & 20 & 0.170 & pilotnov & 18 & 0.103 & stair & 13 & 0.035 \\
degen2 & 17 & 0.053 & qap8 & 10 & 0.116 & standata & 16 & 0.012 \\
degen3 & 39(1) & 0.822 & qap12 & 16 & 3.341 & standgub & 16 & 0.012 \\
dfl001 & 61 & 4.233 & qap15 & 18 & 10.291 & standmps & 24 & 0.029 \\
e226 & 18 & 0.019 & recipe & 10 & 0.001 & stocfor1 & 12 & 0.003 \\
etamacro & 20 & 0.030 & sc105 & 12 & 0.002 & stocfor2 & 20 & 0.264 \\
fffff800 & 30 & 0.055 & sc205 & 13 & 0.007 & stocfor3 & 29 & 3.355 \\
finnis & 22 & 0.030 & sc50a & 12 & 0.001 & tuff & 22 & 0.023 \\
fit1d & 17 & 0.016 & sc50b & 9 & 0.001 & truss & 17 & 0.165 \\
fit1p & 15 & 0.112 & scagr25 & 26 & 0.020 & vtp.base & 12 & 0.001 \\
fit2d & 37 & 0.242 & scagr7 & 18 & 0.003 & wood1p & 24 & 0.096 \\
fit2p & 24 & 1.181 & scfxm1 & 20 & 0.020 & woodw & 41 & 0.167 \\
forplan & 27(2) & 0.021 & scfxm2 & 23 & 0.047 & & & \\
\bottomrule
\end{tabular}
\end{table}

\subsection{Results for PyPSA problems}
The results on the collection of PyPSA problems are reported in Table~\ref{tab:pypsa_results}. For the sake of space, the name of the problems is shortened;  the full name of the problems, as well as their size, are reported in Table~\ref{tab:problem_data}.

For small instances, IPX is clearly the winner; however, for larger sizes, HiPO outperforms IPX by one order of magnitude. Notice that for all these problems, the augmented system was chosen (and indeed it outperformed the normal equations approach).
The backup of using the basis preconditioner at the end was triggered only in two cases and it required a single IPX iteration.

\begin{table}
\tiny
\caption{Results on the 45 PyPSA problems. The column ``Iter'' for HiPO shows also in parentheses the number of IPX iterations performed at the end (if any).}
\label{tab:pypsa_results}
\centering
\begin{tabular}{l|rr|rr||l|rr|rr}
\toprule
& \multicolumn{2}{c|}{IPX} & \multicolumn{2}{c||}{HiPO} & & \multicolumn{2}{c|}{IPX} & \multicolumn{2}{c}{HiPO} \\
\midrule
Problem & Iter & Time & Iter & Time & Problem & Iter & Time & Iter & Time\\
\midrule
trex-2-24h & 41 & 3.1 & 33 & 10.4 & trex-10-24h & 56 & 61.8 & 55 & 48.9 \\
op-2-24h & 38 & 2.5 & 31 & 7.5 & op-10-24h & 59 & 61.9 & 59 & 52.9 \\
trex-3-24h & 46 & 7.2 & 42 & 14.4 & op-5-12h & 54 & 64.4 & 55 & 52.1 \\
op-3-24h & 42 & 6.0 & 38 & 12.7 & trex-5-12h & 64 & 64.7 & 60 & 65.0 \\
trex-2-12h & 48 & 10.9 & 35 & 14.7 & trex-6-12h & 67 & 96.7 & 67 & 78.4 \\
op-2-12h & 40 & 8.9 & 32 & 13.4 & op-6-12h & 57 & 84.5 & 61 & 66.6 \\
trex-4-24h & 51 & 12.5 & 45 & 23.3 & trex-2-3h & 52 & 128.1 & 35 & 65.4 \\
op-4-24h & 44 & 11.2 & 42 & 20.7 & op-2-3h & 46 & 99.7 & 33 & 55.6 \\
op-5-24h & 48 & 16.8 & 45 & 23.6 & op-7-12h & 60 & 129.6 & 65 & 84.9 \\
trex-5-24h & 54 & 19.1 & 48 & 28.5 & trex-7-12h & 71 & 137.3 & 70 & 86.5 \\
trex-3-12h & 51 & 25.6 & 47 & 39.9 & op-8-12h & 64 & 164.5 & 70 & 89.0 \\
op-3-12h & 47 & 22.4 & 39 & 30.5 & trex-8-12h & 71 & 181.9 & 72 & 91.4 \\
op-6-24h & 47 & 24.5 & 47 & 27.4 & op-9-12h & 65 & 193.2 & 69 & 106.4 \\
trex-6-24h & 56 & 25.4 & 53 & 34.7 & trex-9-12h & 72 & 195.0 & 75 & 103.1 \\
sec-2-24h & 51 & 37.9 & 51 (1) & 29.7 & op-10-12h & 68 & 251.9 & 69 & 109.2 \\
op-7-24h & 51 & 34.7 & 49 & 36.9 & trex-10-12h & 72 & 284.0 & 72 & 111.6 \\
trex-7-24h & 57 & 34.4 & 58 & 42.2 & trex-3-3h & 54 & 304.1 & 44 & 87.4 \\
op-4-12h & 51 & 38.2 & 51 & 47.3 & op-3-3h & 51 & 211.2 & 39 & 74.0 \\
trex-4-12h & 57 & 49.2 & 51 & 57.4 & trex-4-3h & 62 & 573.2 & 52 & 131.1 \\
op-8-24h & 49 & 40.8 & 50 & 37.7 & op-4-3h & 57 & 484.1 & 46 & 116.7 \\
trex-8-24h & 58 & 43.2 & 57 & 46.9 & trex-2-1h & 56 & 1230.1 & 62 & 217.8 \\
op-9-24h & 56 & 45.1 & 55 & 51.5 & op-2-1h & 50 & 773.2 & 43 (1) & 158.4 \\
trex-9-24h & 62 & 53.9 & 56 & 50.1 & & & & & \\
\bottomrule
\end{tabular}
\end{table}

In Figure~\ref{fig:logplot_time} we show a logarithmic plot of the computational time taken to solve each PyPSA problem, by the HiGHS simplex solver, IPX and HiPO using normal equations and augmented system. Since it is a logarithmic plot, a straight line with slope $\omega$ corresponds to a polynomial growth with exponent $\omega$. Therefore, we can see that the time taken by the simplex solver grows like $n^{2.94}$, where $n$ is the number of variables of the problem, and it hits the maximum time of $2000s$ for the last couple of problems. The time taken by IPX instead grows like $n^{1.85}$. Using HiPO with normal equations, the growth is similar to IPX, with exponent $1.74$, but the overall time is larger; the matrix becomes very dense and the solver runs out of memory for the largest problems. Using HiPO with augmented system, the time grows like $n^{0.97}$ and the solver becomes undoubtedly the winner for large sizes. Notice also that the largest problems require only a few hundred MB of memory for the whole factorisation when using augmented system, while it would use more than 128GB with normal equations.

It may be surprising that the augmented system solver demonstrates a sub-linear computational complexity. This happens because the solves are  much more expensive than the factorisations for these problems, and because both the augmented system matrix and its factor are very sparse. 
Performing solves has a cost proportional to the number of nonzero entries in the factor, which is proportional to the size of the matrix due to the high degree of sparsity. Since solves are the the most expensive operations, this means that a single IPM iteration has actually a cost proportional to the size of the matrix. Moreover the operations become even more efficient for larger problems, due to properties of the BLAS implementations on modern machines. As a result of that the measured computational complexity is actually slightly sub-linear. This type of behaviour is not to be expected for general linear programs.

\begin{figure}
\caption{Time taken by simplex, IPX and HiPO to solve the PyPSA problems}
\label{fig:logplot_time}
\centering
\includegraphics[width=.85\textwidth]{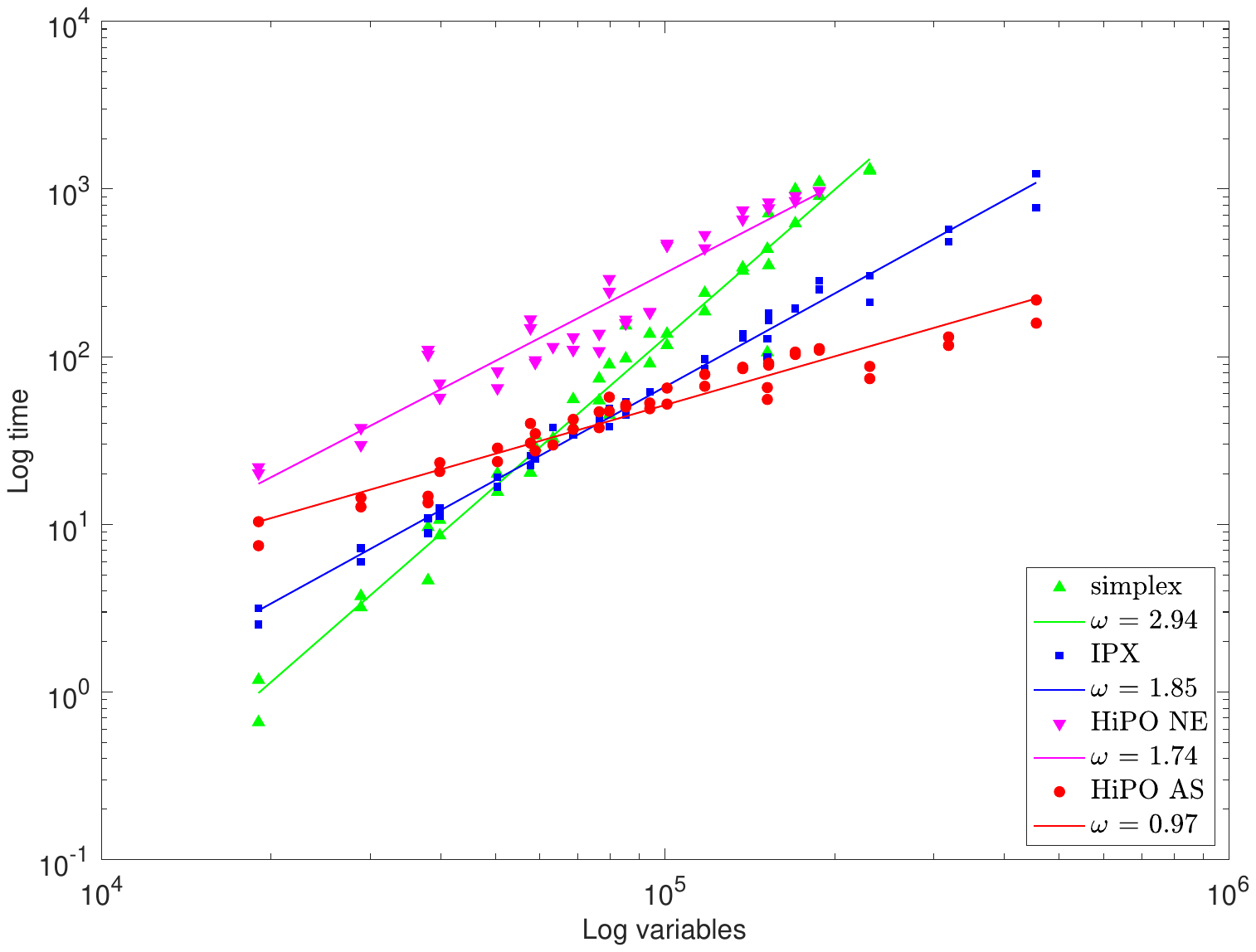}
\end{figure}

\subsection{Time per iteration}
In Figure~\ref{fig:time_per_iter}, the time taken for each IPM iteration by IPX and by HiPO using augmented system is reported, for the largest PyPSA problem \texttt{pypsa-eur-elec-trex-2-1h}.

\begin{figure}
\caption{Time taken by each IPM iteration for a medium size PyPSA problem.}
\label{fig:time_per_iter}
\centering
\includegraphics[width=.85\textwidth]{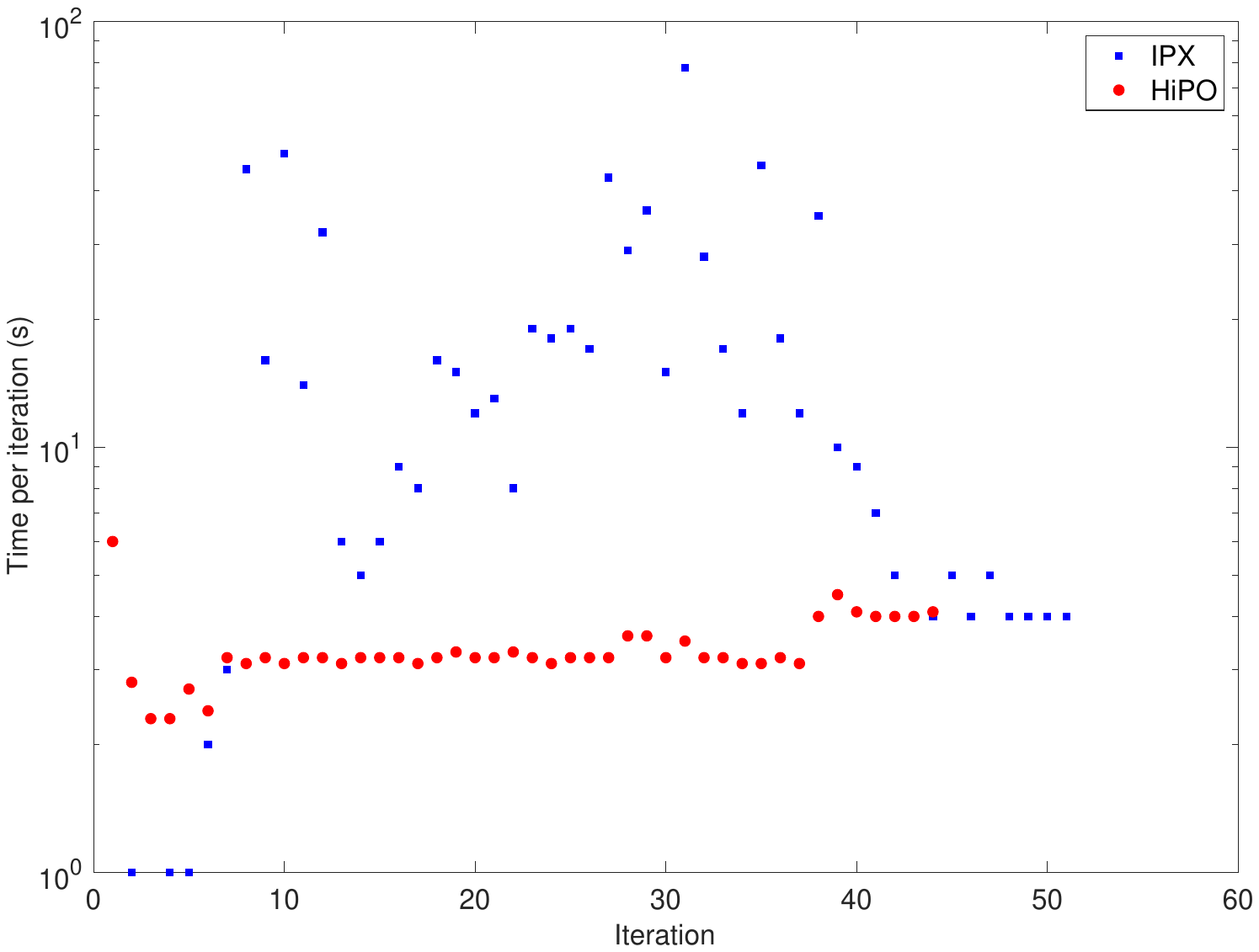}
\end{figure}

We can identify three phases of IPX:
\begin{itemize}
\item At the beginning, a diagonal preconditioner is used, which is inexpensive to compute and apply. The time per iteration is very low.
\item In the middle, a basis preconditioner is computed and the basis keeps changing substantially at each iteration. This requires expensive operations to update the factorisation of the basis. Moreover, the guessed basis is likely to be a poor approximation of the optimal one, since the IPM iterations are too far from optimality, leading to a large number of Krylov iterations. Thus, the time per iteration is very high.
\item At the end, the basis has stabilised, the preconditioner does not need to be recomputed often and fewer Krylov iterations are needed since the basis is already close to the optimal one. The time per iteration becomes small again.
\end{itemize}

We can identify three phases of HiPO as well:
\begin{itemize}
\item The first iteration appears to be more expensive, because it includes the time taken for the ordering and symbolic factorisation. 
\item In the middle, the time per iteration is very consistent, since the factorisation performs always the same number of operations, and there is no dynamic pivoting performed. The number of solves performed is also almost constant, since the linear systems are not too ill conditioned and only few iterative refinement steps are needed.
\item At the end, the linear systems are very ill conditioned and more iterative refinement steps are needed. Thus, the time per iteration increases slightly. 
\end{itemize}

The fact that the time per iteration towards the end appears to be the same for both methods is just a coincidence. The relevant observation to make is that the time per iteration of IPX has large variations and decreases substantially at the end, while the time per iteration of HiPO is almost constant, and considerably lower than that of IPX.

The fact that IPX has a low and stable time per iteration towards the end motivates the choice of using it as a backup solver: if HiPO struggles with reaching the required accuracy, but still produces a point close to the optimal vertex, then executing one or two IPX iterations is expected to complete the optimization process, i.e., reach a highly accurate solution. These iterations are likely to be cheap, since the point is already very close to being optimal. However, the cost of identifying and factorising the basis may be substantial, especially if crossover is not requested; therefore, this should be regarded only as a backup plan in the (rare) cases in which HiPO struggles, rather than as a hybrid strategy to solve linear programs.

\subsection{Results for Mittelmann collection}
The results on the Mittelmann collection are shown in Table~\ref{tab:mittelmann_results}. These problems are much larger and harder to solve than the previous ones, to the point that many commercial solvers also struggle to reach the required optimality criteria. In the table, $\dagger$ indicates that the solver reached the time limit of $5000$ seconds; $\ddagger$ indicates that the solver ran out of memory; $\ast$ indicates that the solver failed for other reasons, the most common being stagnation or internal error.

Notice that the two solvers have slightly different ways of pre- and post-processing the problem; this can lead to a solution being declared {\it imprecise} (i.e., an optimal solution that is not optimal any more after post-processing) by a solver, that would be accepted as optimal by the other, and vice versa. Given the different strategies used by the two solvers, any problem declared imprecise is considered as optimal in the statistics of Table~\ref{tab:mittelmann_results}. Notice that this only affects a couple of problems.

\begin{table}
\tiny
\caption{Results on the Mittelmann problems. The column ``Iter'' for HiPO shows also in parentheses the number of IPX iterations performed at the end (if any). $\dagger$ indicates time limit, $\ddagger$ indicates out of memory, $\ast$ indicates other failures.}
\label{tab:mittelmann_results}
\centering
\begin{tabular}{l|rr|rr||l|rr|rr}
\toprule
& \multicolumn{2}{c|}{IPX} & \multicolumn{2}{c||}{HiPO} & & \multicolumn{2}{c|}{IPX} & \multicolumn{2}{c}{HiPO} \\
\midrule
Problem & Iter & Time & Iter & Time & Problem & Iter & Time & Iter & Time\\
\midrule
a2864 & 23 & 301.9 & 12 & 513.2 & nug08-3rd & 22 & 334.3 & 11 & 34.3 \\
bdry2 &  & $\dagger$ &  & $\ast$ & pds-100 & 81 & 79.2 & 47 & 59.0 \\
cont1 & 20 & 84.0 & 22(3) & 695.8 & psched3-3 & 57 & 22.9 & 32 & 139.7 \\
cont11 & 23 & 542.0 & 47 & 55.8 & Primal2\_1000 & 40 & 742.0 & 23 & 190.2 \\
datt256 & 23 & 120.7 & 9 & 5.4 & qap15 & 30 & 14.9 & 18 & 9.8 \\
degme & 50 & 2600.6 & 32 & 255.86 & rail02 & 72 & 49.3 & 65 & 106.0 \\
dlr1 & 98 & 1048.6 & 89(1) & 377.3 & rail4284 & 104 & 329.8 & 42(2) & 119.3 \\
dlr2 &  & $\dagger$ &  & $\dagger$ & rmine15 & 53 & 167.0 & 29 & 374.5 \\
Dual2\_5000 &  & $\dagger$ &  & $\dagger$ & s100 & 79 & 66.0 & 60(1) & 27.7 \\
ex10 & 16 & 11.2 & 16 & 70.6 & s250r10 & 94 & 63.2 & 38 & 25.1 \\
fhnw-bin1 & 29 & 535.0 & 26 & 825.0 & s82 & 144 & 1459.7 & 97(7) & 267.9 \\
fome13 & 43 & 25.7 & 56(3) & 20.7 & savsched1 & 28 & 39.1 & 26(5) & 90.5 \\
graph40-40 & 15 & 6.8 & 11 & 57.0 & scpm1 & 28 & 49.4 & 26 & 125.4 \\
irish-e & 60 & 27.4 & 62 & 38.1 & set-cover & 57 & 889.6 & 27 & 661.0 \\
L1\_sixm1000obs &  & $\ast$ & 16(59) & 1673.0 & shs1023 & 135 & 194.4 &  & $\ast$ \\
L1\_sixm250obs &  & $\ast$ & 20(53) & 129.0 & square41 & 28 & 70.8 & 19(3) & 51.8 \\
L2CTA3D &  & $\ast$ &  & $\ddagger$ & stat96v2 &  & $\ast$ & 24(58) & 311.2 \\
Linf\_520c &  & $\ast$ &  & $\ast$ & stormG2\_1000 &  & $\ast$ &  & $\ast$ \\
neos & 96 & 91.7 & 79 & 311.9 & stp3d & 64 & 183.7 & 35 & 39.6 \\
neos-3025225 & 33 & 87.2 & 19 & 227.1 & support10 & 32 & 3.7 & 20 & 42.7 \\
neos-5052403 & 28 & 22.4 & 34(1) & 28.7 & thk\_48 &  & $\dagger$ &  & $\ddagger$\\
neos-5251015 & 11 & 145.2 & 21(7) & 973.8 & thk\_63 &  & $\dagger$ &  & $\ddagger$ \\
neos3 & 20 & 22.9 & 10 & 213.2 & tpl-tub-ws16 & 111 & 788.0 & 134(3) & 468.0 \\
ns1687037 &  & $\ast$ & 19 & 101.8 & woodlands09 & 27 & 223.2 & 16 & 33.8 \\
ns1688926 &  & $\ast$ &  & $\ast$  \\
\bottomrule
\end{tabular}
\end{table}

Sometimes, HiPO significantly outperforms IPX, like in the case of \texttt{cont11} where the time is $10$ times smaller; some other times, HiPO is much slower than IPX, like in the case of \texttt{support10}, where the time is $10$ times larger. Overall, considering the problems where both solvers found a solution, in 11 instances HiPO was at least two times faster than IPX, in 7 instances it was between 1 and 2 times faster, in 4 instances IPX was between 1 and 2 times faster than HiPO, and in 13 instances IPX was at least 2 times faster.

This is not surprising, since the two solvers use completely different strategies to solve the linear systems. It is for this reason that both solvers will remain available in HiGHS, rather than the new one replacing the old one; there are cases where a Krylov solver is better than a direct factorisation. Notice that around a quarter of the problems were solved using the augmented system, while for the remaining ones it was better to use the normal equations.

Most of the problems do not need additional iterations with the basis preconditioner at the end. A few problems need only a few of such iterations. Three problems need many more iterations: \texttt{L1\_sixm1000obs}, \texttt{L1\_sixm250obs}, \texttt{stat96v2}; for these problems, HiPO is essentially computing a very expensive starting point for IPX. This situation is not ideal and it will require time and effort to understand the reasons of the numerical challenges that these problems bring. It is certainly on a ``to do'' list of the Authors, but the outcome of such effort is hard to predict and the Authors did not want to delay the release of the software. 
It is important to notice that these problems are extremely challenging and that IPX alone is not able to converge for these instances.

For three problems, the factorisation code runs out of memory: \texttt{L2CTA3D}, \texttt{thk\_48}, \texttt{thk\_63}. For these problems, the factor $L$ has many billion nonzero entries; for the last two instances, 32 bit integers are not enough to store the factor, since there is integer overflow in the arrays used to store the sparsity pattern.

Overall, IPX solves 36 problems and HiPO solves 39. To evaluate the relative performance of the two solvers, we use the {\it shifted geometric mean}, in accordance with the Mittelmann benchmark \cite{mittelmann}: the shifted geometric mean of a sequence of times is given by 
\[G^\sigma\big((t_i)_{i=1,\dots,n}\big)\,\coloneqq\,\sqrt[n]{\prod_{i=1}^n (t_i+\sigma)}\,\,-\,\,\sigma\,\, = \,\,\exp\Big(\sum_{i=1}^n \log(t_i+\sigma)/n\Big) - \sigma,\]
where the formula on the right is just a more efficient way of computing it. $\sigma$ is the shift, which is set to $10$ seconds. The advantage of using shifted geometric mean is that it is not affected much by large or small outliers. 
If a solver failed to compute a solution for a given problem, the corresponding time $t_i$ is set to the time limit of $5000$ seconds. We compute three means: 
\begin{itemize}
\item Considering all 49 problems, IPX has a mean of $338$ seconds and HiPO of $274$ seconds.
\item Considering only the 40 problems for which at least one of the solvers found a solution, IPX has a mean of $181$ seconds and HiPO of $139$ seconds.
\item Considering only the 35 problems for which both solvers found a solution, IPX has a mean of $121$ seconds and HiPO of $114$ seconds.
\end{itemize}

We can see that all three ways of computing the mean give a slight advantage to HiPO over IPX. However, it should be noticed that IPX has the advantage of providing a basis together with the optimal solution,  even if crossover is not performed. For some applications, this is an important feature, worth spending the extra time. 

Considering the best of the two solvers for each problem, the shifted geometric mean over all the 49 problems becomes $175$ seconds, which is much less than that of IPX or HiPO. A heuristic strategy capable of selecting the best solver between IPX and HiPO could potentially improve significantly the running time of HiGHS.

Notice that, unless the solver runs out of memory, some iterations are performed even if the solver fails to find the optimal solution. Therefore, the next two sections analyse the performance of an average iteration, for the $46$ instances that did not run out of memory.

\subsection{Effect of multi-threading}

In Figure~\ref{fig:speedup}, the time taken to perform the factorisation alone is shown, for 46 instances in the Mittelmann collection, in serial and in parallel (8 cores). The solve time is not included because it is not parallelised. The problems are sorted according to the serial factorisation time. The first dotted line from the top indicates a speed-up of 2 times compared to the serial time. The other lines indicate a speed-up of 4, 6 and 8 times. 

It can be seen that not all problems benefit from parallelism. For some instances, the overhead of parallelisation actually slows down the solver. Many problems however achieve speed-up larger than 2 and a few problems larger than 4. Given the nature of the factorisation algorithm and the unstructured constraint matrices in the challenging Mittelmann collection, it is unrealistic to achieve large speed-ups. Moreover, the fill-reducing ordering, symbolic factorisation and solve phases do not run in parallel, so an improvement in the parallelisation of the factorisation phase may have only a small impact on the overall time.

Notice that the speed up factor depends on the specific architecture and BLAS library used. The same experiments run on the Author's Macbook, using Apple's BLAS implementation, yield different results, with more problems achieving a large speedup.

\begin{figure}
\caption{Time taken by factorisation alone for the Mittelmann collection, running in serial and in parallel (8 threads). The dotted lines indicate a speedup of 2x, 4x, 6x, 8x compared to the serial time.}
\label{fig:speedup}
\centering
\includegraphics[width=.8\textwidth]{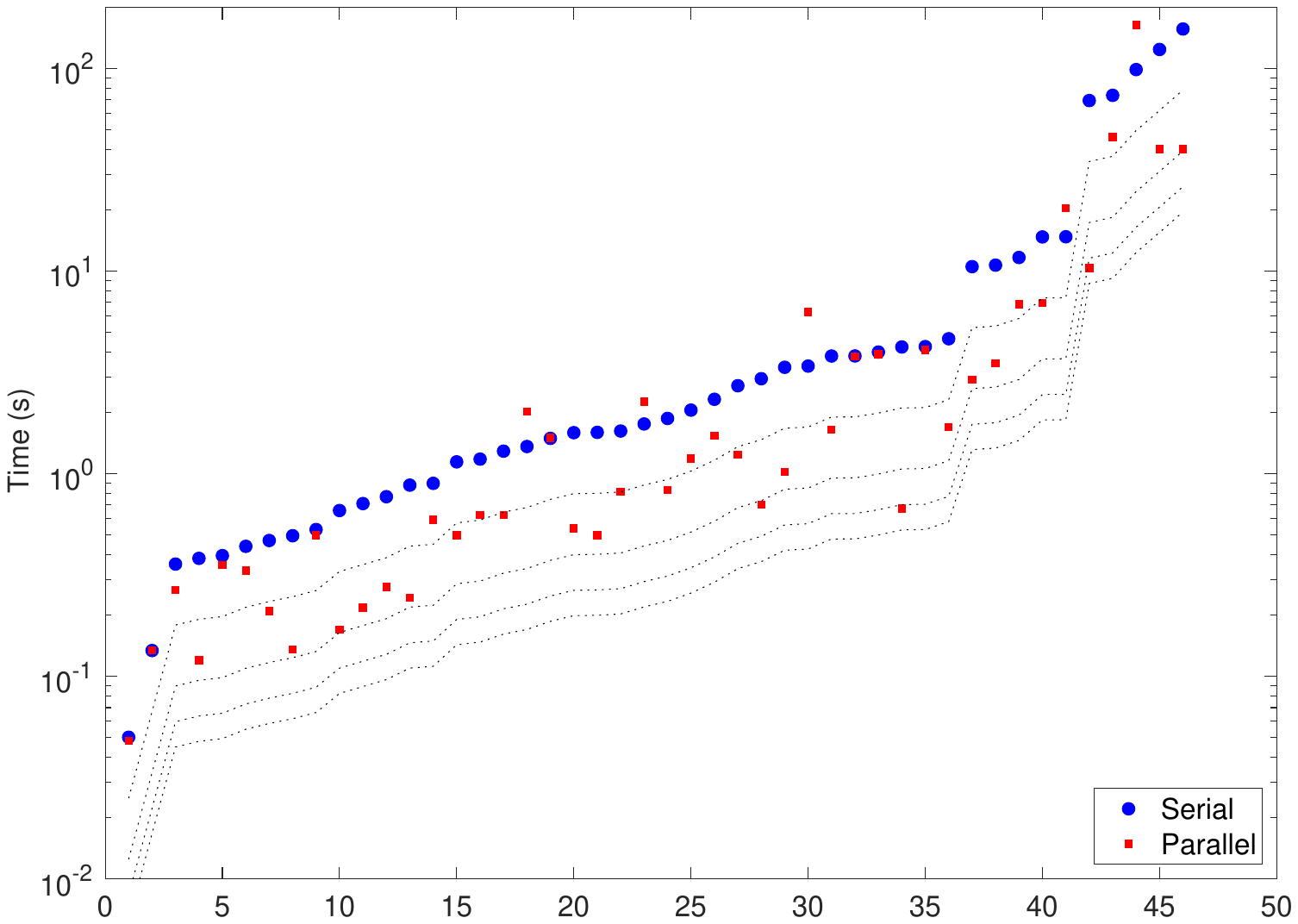}
\end{figure}

\subsection{Time fraction of various phases}

In Figure~\ref{fig:time_fraction}, we show the fraction of time spent forming the matrix (either normal equations or augmented system), factorising it and performing the solves. For most problems, the predominant task is the factorisation; however, when the normal equations approach is selected, simply forming the matrix may take a considerable amount of time. The solve phase can also be quite expensive. 

Forming the normal equations is a task that can be executed in parallel, even though the current code does not implement this. The solves are much harder to parallelise; moreover, each solve involved in centrality correctors or iterative refinement needs to wait for the previous one to finish, since they depend on one another. The fill-reducing ordering and symbolic factorisation are not shown in this graph, since they do not get executed at each iteration; they can however be quite costly for some problems and they also cannot be parallelised easily. 

Even though the factorisation is the phase that performs the largest number of floating point operations, it is also the phase that benefits most from the parallelisation and it can exploit higher level BLAS functions. For these reasons, the factorisation phase may not be the main bottleneck for some problems. Much attention should be put in how to parallelise the solves and how to compute the normal equations matrix efficiently. For some problems, the normal equations matrix is so dense that treating it as a sparse matrix may not be the best option.

Out of the 17 problems for which IPX was faster then HiPO in Table~\ref{tab:mittelmann_results}, 3 of them suffered from a particularly slow computation of normal equations, and 14 suffered from particularly slow solves. This means that there is much performance to be gained by improving these areas.

\begin{figure}
\caption{Fraction of time taken to form the matrix, factorise it and perform solves, for the Mittelmann collection.}
\label{fig:time_fraction}
\centering
\includegraphics[width=.8\textwidth]{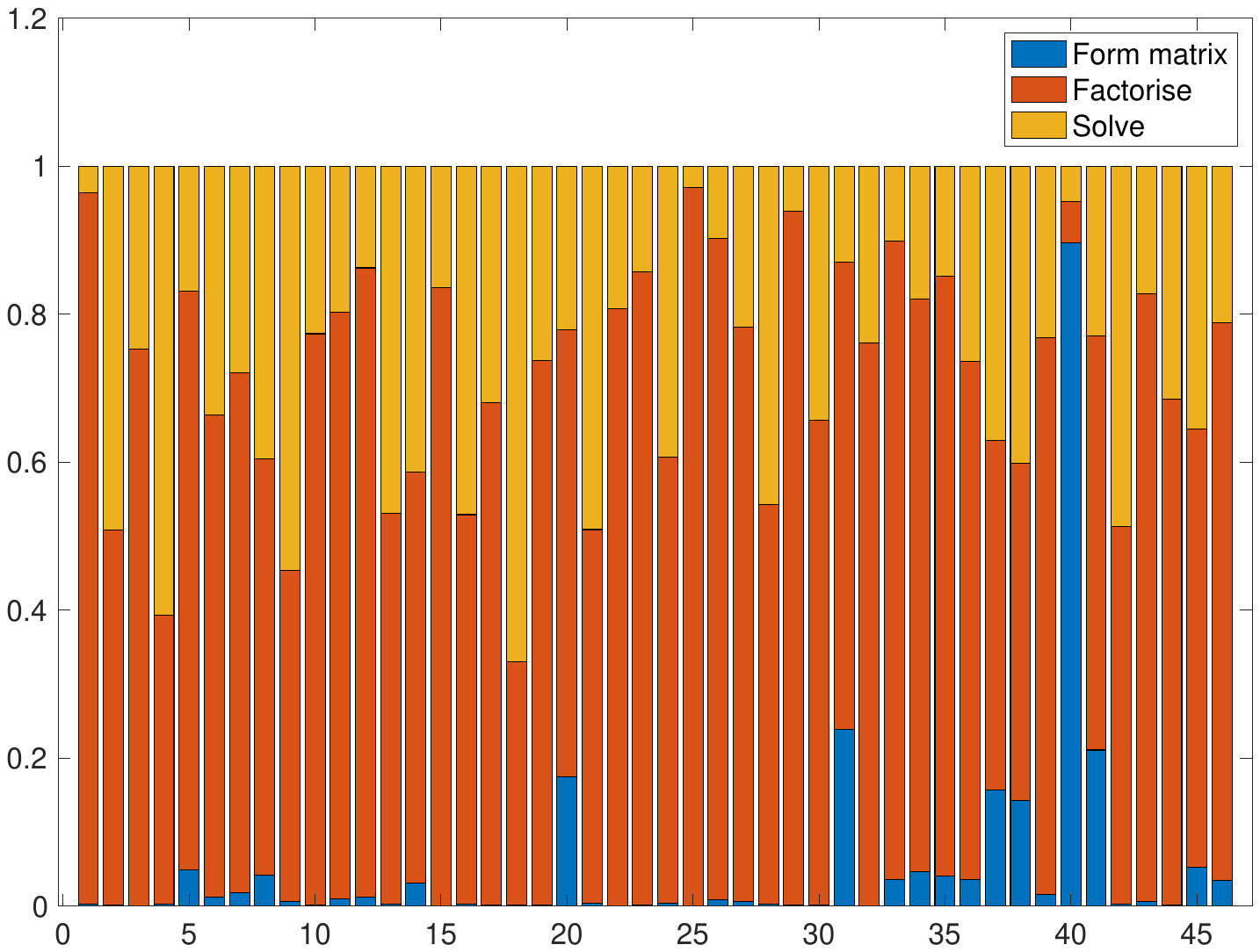}
\end{figure}

\section{Conclusion}

We have described the details of the implementation of a new factorisation-based interior point method for the HiGHS library and made the whole code fully accessible to the optimization community. The experiments show that HiPO is a high performance implementation, capable of outperforming the current HiGHS interior point solver IPX. In particular, the results on linear programs coming from energy modelling highlight the significant advantages of factorising the augmented system rather than the normal equations matrix. Future improvements to the solver will include:
\begin{itemize}
\item Enhancing the parallelisation of the code, using multi-threading also while forming the normal equations and performing solves.
\item Improving the way in which memory is allocated and managed for the various stages of the factorisation.
\item Improving the accuracy of the solver, through better regularisation and pivoting strategies.
\item Including support for quadratic programming.
\end{itemize}

\section*{Acknowledgements}
The Authors are grateful to Jennifer Scott for providing advice regarding the implementation of the multifrontal factorisation.

\bibliographystyle{siam}
\bibliography{biblio}

\appendix
\section{Problem's size}
\label{apx:size}

{\tiny
\begin{longtable}{l|rrr|rrr}
\caption{Size of original and presolved problems.}\label{tab:problem_data}\\
\hline
\multicolumn{7}{c}{MITTELMANN COLLECTION} \\
\hline
& \multicolumn{3}{c|}{Original} & \multicolumn{3}{c}{Presolved} \\
\hline
Problem & rows & cols & nz & rows & cols & nz \\
\hline
a2864 & 2.2e+04 & 2.0e+05 & 2.0e+07 & 2.1e+04 & 1.4e+04 & 1.4e+06 \\
bdry2 & 3.8e+05 & 2.5e+05 & 1.5e+06 & 2.9e+05 & 1.7e+05 & 1.8e+06 \\
cont1 & 1.6e+05 & 4.0e+04 & 4.0e+05 & 1.2e+05 & 4.0e+04 & 3.6e+05 \\
cont11 & 1.6e+05 & 8.0e+04 & 4.0e+05 & 1.2e+05 & 8.0e+04 & 3.6e+05 \\
datt256 & 1.1e+04 & 2.6e+05 & 1.5e+06 & 9.9e+03 & 2.0e+05 & 1.1e+06 \\
degme & 1.9e+05 & 6.6e+05 & 8.1e+06 & 1.9e+05 & 6.6e+05 & 8.1e+06 \\
dlr1 & 1.7e+06 & 9.1e+06 & 1.8e+07 & 4.7e+05 & 5.2e+05 & 2.2e+06 \\
dlr2 & 7.1e+06 & 3.9e+07 & 7.8e+07 & 2.0e+06 & 2.2e+06 & 9.3e+06 \\
Dual2\_5000 & 3.0e+07 & 3.3e+07 & 9.3e+07 & 3.0e+07 & 3.3e+07 & 9.3e+07 \\
ex10 & 7.0e+04 & 1.8e+04 & 1.2e+06 & 6.3e+04 & 1.6e+04 & 1.0e+06 \\
fhnw-bin1 & 7.7e+05 & 1.1e+06 & 8.6e+06 & 7.7e+05 & 1.1e+06 & 8.6e+06 \\
fome13 & 4.9e+04 & 9.8e+04 & 2.9e+05 & 2.9e+04 & 7.5e+04 & 2.5e+05 \\
graph40-40 & 3.6e+05 & 1.0e+05 & 1.3e+06 & 3.4e+05 & 9.8e+04 & 1.2e+06 \\
irish-e & 1.0e+05 & 6.2e+04 & 5.2e+05 & 6.5e+04 & 3.7e+04 & 4.0e+05 \\
L1\_sixm1000obs & 3.1e+06 & 1.4e+06 & 1.4e+07 & 5.1e+05 & 1.0e+06 & 2.2e+06 \\
L1\_sixm250obs & 9.9e+05 & 4.3e+05 & 4.3e+06 & 1.5e+05 & 3.1e+05 & 6.4e+05 \\
L2CTA3D & 2.1e+05 & 1.0e+07 & 3.0e+07 & 2.1e+05 & 8.0e+06 & 2.4e+07 \\
Linf\_520c & 9.3e+04 & 6.9e+04 & 5.7e+05 & 6.1e+04 & 6.2e+04 & 2.8e+05 \\
neos & 4.8e+05 & 3.7e+04 & 1.0e+06 & 4.2e+05 & 3.6e+04 & 9.1e+05 \\
neos-3025225 & 9.2e+04 & 7.0e+04 & 9.4e+06 & 8.1e+04 & 7.0e+04 & 4.9e+06 \\
neos-5052403 & 3.8e+04 & 3.3e+04 & 4.9e+06 & 1.9e+04 & 2.8e+04 & 2.4e+06 \\
neos-5251015 & 4.9e+05 & 1.4e+05 & 2.0e+06 & 3.9e+05 & 1.1e+05 & 1.7e+06 \\
neos3 & 5.1e+05 & 6.6e+03 & 1.5e+06 & 5.1e+05 & 6.6e+03 & 1.5e+06 \\
ns1687037 & 5.1e+04 & 4.4e+04 & 1.4e+06 & 3.6e+04 & 3.1e+04 & 1.4e+06 \\
ns1688926 & 3.3e+04 & 1.7e+04 & 1.7e+06 & 2.5e+04 & 1.6e+04 & 9.0e+05 \\
nug08-3rd & 2.0e+04 & 2.0e+04 & 1.4e+05 & 1.8e+04 & 2.0e+04 & 1.3e+05 \\
pds-100 & 1.6e+05 & 5.1e+05 & 1.1e+06 & 7.8e+04 & 4.2e+05 & 9.6e+05 \\
psched3-3 & 2.7e+05 & 8.0e+04 & 1.1e+06 & 2.0e+04 & 1.1e+04 & 2.0e+05 \\
Primal2\_1000 & 1.3e+06 & 2.6e+06 & 5.5e+06 & 6.8e+05 & 1.3e+06 & 3.6e+06 \\
qap15 & 6.3e+03 & 2.2e+04 & 9.5e+04 & 5.7e+03 & 2.2e+04 & 8.5e+04 \\
rail02 & 9.6e+04 & 2.7e+05 & 7.6e+05 & 5.0e+04 & 1.9e+05 & 6.0e+05 \\
rail4284 & 4.3e+03 & 1.1e+06 & 1.1e+07 & 4.2e+03 & 1.1e+06 & 1.1e+07 \\
rmine15 & 3.6e+05 & 4.2e+04 & 8.8e+05 & 3.6e+05 & 4.2e+04 & 8.8e+05 \\
s100 & 1.5e+04 & 3.6e+05 & 1.8e+06 & 1.4e+04 & 3.6e+05 & 1.4e+06 \\
s250r10 & 1.1e+04 & 2.7e+05 & 1.3e+06 & 6.6e+03 & 2.7e+05 & 1.2e+06 \\
s82 & 8.8e+04 & 1.7e+06 & 7.0e+06 & 8.1e+04 & 1.7e+06 & 6.8e+06 \\
savsched1 & 3.0e+05 & 3.3e+05 & 1.8e+06 & 2.9e+05 & 3.1e+05 & 1.7e+06 \\
scpm1 & 5.0e+03 & 5.0e+05 & 6.2e+06 & 5.0e+03 & 5.0e+05 & 6.2e+06 \\
set-cover & 1.0e+04 & 1.1e+06 & 2.0e+07 & 1.0e+04 & 1.1e+06 & 2.0e+07 \\
shs1023 & 1.3e+05 & 4.4e+05 & 1.0e+06 & 1.3e+05 & 4.3e+05 & 9.5e+05 \\
square41 & 4.0e+04 & 6.2e+04 & 1.4e+07 & 1.8e+03 & 2.4e+04 & 4.3e+06 \\
stat96v2 & 2.9e+04 & 9.6e+05 & 2.9e+06 & 2.2e+04 & 9.3e+05 & 2.8e+06 \\
stormG2\_1000 & 5.3e+05 & 1.3e+06 & 3.3e+06 & 3.8e+05 & 1.1e+06 & 2.9e+06 \\
stp3d & 1.6e+05 & 2.0e+05 & 6.6e+05 & 1.4e+05 & 1.8e+05 & 5.8e+05 \\
support10 & 1.7e+05 & 1.5e+04 & 5.6e+05 & 1.1e+05 & 9.0e+03 & 3.6e+05 \\
thk\_48 & 6.4e+06 & 8.6e+06 & 2.8e+07 & 6.0e+06 & 8.2e+06 & 2.6e+07 \\
thk\_63 & 5.7e+06 & 7.7e+06 & 2.2e+07 & 5.2e+06 & 7.1e+06 & 2.0e+07 \\
tpl-tub-ws16 & 1.2e+06 & 7.5e+05 & 4.7e+06 & 3.2e+05 & 7.1e+05 & 2.5e+06 \\
woodlands09 & 1.9e+05 & 3.8e+05 & 2.6e+06 & 4.2e+04 & 2.3e+05 & 2.4e+06 \\
\hline
\hline
\multicolumn{7}{c}{PYPSA-EUR PROBLEMS} \\
\hline
& \multicolumn{3}{c|}{Original} & \multicolumn{3}{c}{Presolved} \\
\hline
Problem & rows & cols & nz & rows & cols & nz \\
\hline
pypsa-eur-elec-trex-2-24h & 3.9e+04 & 1.9e+04 & 7.2e+04 & 1.3e+04 & 1.7e+04 & 4.5e+04 \\
pypsa-eur-elec-op-2-24h & 3.9e+04 & 1.9e+04 & 7.1e+04 & 1.3e+04 & 1.7e+04 & 4.3e+04 \\
pypsa-eur-elec-trex-3-24h & 6.0e+04 & 2.9e+04 & 1.1e+05 & 2.2e+04 & 2.5e+04 & 7.3e+04 \\
pypsa-eur-elec-op-3-24h & 6.0e+04 & 2.9e+04 & 1.1e+05 & 2.1e+04 & 2.5e+04 & 6.9e+04 \\
pypsa-eur-elec-trex-2-12h & 7.7e+04 & 3.8e+04 & 1.4e+05 & 2.7e+04 & 3.3e+04 & 8.9e+04 \\
pypsa-eur-elec-op-2-12h & 7.7e+04 & 3.8e+04 & 1.4e+05 & 2.5e+04 & 3.3e+04 & 8.6e+04 \\
pypsa-eur-elec-trex-4-24h & 8.5e+04 & 4.0e+04 & 1.7e+05 & 3.2e+04 & 3.5e+04 & 1.1e+05 \\
pypsa-eur-elec-op-4-24h & 8.5e+04 & 4.0e+04 & 1.6e+05 & 3.1e+04 & 3.4e+04 & 1.0e+05 \\
pypsa-eur-elec-op-5-24h & 1.1e+05 & 5.0e+04 & 2.1e+05 & 4.2e+04 & 4.3e+04 & 1.3e+05 \\
pypsa-eur-elec-trex-5-24h & 1.1e+05 & 5.0e+04 & 2.1e+05 & 4.2e+04 & 4.4e+04 & 1.4e+05 \\
pypsa-eur-elec-trex-3-12h & 1.2e+05 & 5.8e+04 & 2.3e+05 & 4.4e+04 & 5.0e+04 & 1.4e+05 \\
pypsa-eur-elec-op-3-12h & 1.2e+05 & 5.8e+04 & 2.2e+05 & 4.2e+04 & 5.0e+04 & 1.4e+05 \\
pypsa-eur-elec-op-6-24h & 1.3e+05 & 5.9e+04 & 2.4e+05 & 5.0e+04 & 4.9e+04 & 1.6e+05 \\
pypsa-eur-elec-trex-6-24h & 1.3e+05 & 5.9e+04 & 2.5e+05 & 5.3e+04 & 5.1e+04 & 1.7e+05 \\
pypsa-eur-sec-2-24h & 1.3e+05 & 6.3e+04 & 2.9e+05 & 4.7e+04 & 4.6e+04 & 1.8e+05 \\
pypsa-eur-elec-op-7-24h & 1.5e+05 & 6.9e+04 & 2.9e+05 & 6.1e+04 & 5.8e+04 & 1.9e+05 \\
pypsa-eur-elec-trex-7-24h & 1.5e+05 & 6.9e+04 & 3.0e+05 & 6.2e+04 & 5.9e+04 & 2.0e+05 \\
pypsa-eur-elec-op-4-12h & 1.7e+05 & 8.0e+04 & 3.2e+05 & 6.2e+04 & 6.7e+04 & 2.0e+05 \\
pypsa-eur-elec-trex-4-12h & 1.7e+05 & 8.0e+04 & 3.3e+05 & 6.4e+04 & 6.9e+04 & 2.1e+05 \\
pypsa-eur-elec-op-8-24h & 1.6e+05 & 7.6e+04 & 3.2e+05 & 6.5e+04 & 6.5e+04 & 2.1e+05 \\
pypsa-eur-elec-trex-8-24h & 1.6e+05 & 7.6e+04 & 3.2e+05 & 6.8e+04 & 6.6e+04 & 2.2e+05 \\
pypsa-eur-elec-op-9-24h & 1.8e+05 & 8.5e+04 & 3.5e+05 & 7.3e+04 & 7.2e+04 & 2.3e+05 \\
pypsa-eur-elec-trex-9-24h & 1.8e+05 & 8.5e+04 & 3.6e+05 & 7.6e+04 & 7.4e+04 & 2.4e+05 \\
pypsa-eur-elec-trex-10-24h & 2.0e+05 & 9.4e+04 & 4.0e+05 & 8.5e+04 & 8.1e+04 & 2.8e+05 \\
pypsa-eur-elec-op-10-24h & 2.0e+05 & 9.4e+04 & 3.9e+05 & 8.3e+04 & 7.9e+04 & 2.6e+05 \\
pypsa-eur-elec-op-5-12h & 2.2e+05 & 1.0e+05 & 4.2e+05 & 8.4e+04 & 8.5e+04 & 2.6e+05 \\
pypsa-eur-elec-trex-5-12h & 2.2e+05 & 1.0e+05 & 4.3e+05 & 8.5e+04 & 8.7e+04 & 2.8e+05 \\
pypsa-eur-elec-trex-6-12h & 2.5e+05 & 1.2e+05 & 5.0e+05 & 1.0e+05 & 1.0e+05 & 3.3e+05 \\
pypsa-eur-elec-op-6-12h & 2.5e+05 & 1.2e+05 & 4.9e+05 & 1.0e+05 & 9.8e+04 & 3.1e+05 \\
pypsa-eur-elec-trex-2-3h & 3.1e+05 & 1.5e+05 & 5.7e+05 & 1.0e+05 & 1.3e+05 & 3.4e+05 \\
pypsa-eur-elec-op-2-3h & 3.1e+05 & 1.5e+05 & 5.7e+05 & 9.5e+04 & 1.3e+05 & 3.3e+05 \\
pypsa-eur-elec-op-7-12h & 3.0e+05 & 1.4e+05 & 5.8e+05 & 1.2e+05 & 1.1e+05 & 3.8e+05 \\
pypsa-eur-elec-trex-7-12h & 3.0e+05 & 1.4e+05 & 6.0e+05 & 1.2e+05 & 1.2e+05 & 4.0e+05 \\
pypsa-eur-elec-op-8-12h & 3.3e+05 & 1.5e+05 & 6.3e+05 & 1.3e+05 & 1.3e+05 & 4.1e+05 \\
pypsa-eur-elec-trex-8-12h & 3.3e+05 & 1.5e+05 & 6.4e+05 & 1.3e+05 & 1.3e+05 & 4.3e+05 \\
pypsa-eur-elec-op-9-12h & 3.6e+05 & 1.7e+05 & 7.0e+05 & 1.5e+05 & 1.4e+05 & 4.6e+05 \\
pypsa-eur-elec-trex-9-12h & 3.6e+05 & 1.7e+05 & 7.2e+05 & 1.5e+05 & 1.5e+05 & 4.8e+05 \\
pypsa-eur-elec-op-10-12h & 4.1e+05 & 1.9e+05 & 7.9e+05 & 1.7e+05 & 1.6e+05 & 5.2e+05 \\
pypsa-eur-elec-trex-10-12h & 4.1e+05 & 1.9e+05 & 8.0e+05 & 1.7e+05 & 1.6e+05 & 5.5e+05 \\
pypsa-eur-elec-trex-3-3h & 4.8e+05 & 2.3e+05 & 9.0e+05 & 1.7e+05 & 1.9e+05 & 5.5e+05 \\
pypsa-eur-elec-op-3-3h & 4.8e+05 & 2.3e+05 & 8.9e+05 & 1.6e+05 & 1.9e+05 & 5.2e+05 \\
pypsa-eur-elec-trex-4-3h & 6.8e+05 & 3.2e+05 & 1.3e+06 & 2.5e+05 & 2.6e+05 & 8.1e+05 \\
pypsa-eur-elec-op-4-3h & 6.8e+05 & 3.2e+05 & 1.3e+06 & 2.4e+05 & 2.5e+05 & 7.6e+05 \\
pypsa-eur-elec-trex-2-1h & 9.3e+05 & 4.6e+05 & 1.7e+06 & 3.0e+05 & 3.8e+05 & 1.0e+06 \\
pypsa-eur-elec-op-2-1h & 9.3e+05 & 4.6e+05 & 1.7e+06 & 2.8e+05 & 3.8e+05 & 9.7e+05 \\
\hline
\end{longtable}
}

\end{document}